\documentclass{patmorin}
\usepackage{pat}
\usepackage[T1]{fontenc}
\usepackage[utf8]{inputenc}
\usepackage{mathtools}

\usepackage{thmtools}
\usepackage{thm-restate}

\makeatletter
\newcommand*{\greek}[1]{%
  \expandafter\@greek\csname c@#1\endcsname
}
\newcommand*{\@greek}[1]{%
  $\ifcase#1\or\alpha\or\beta\or\gamma\or\delta\or\varepsilon
    \or\zeta\or\eta\or\theta\or\iota\or\kappa\or\lambda
    \or\mu\or\nu\or\xi\or o\or\pi\or\varrho\or\sigma
    \or\tau\or\upsilon\or\phi\or\chi\or\psi\or\omega
    \else\@ctrerr\fi$
}
\makeatother

\usepackage[inline]{enumitem}
\newlist{tightenum}{enumerate}{3}
\setlist[tightenum,1]{noitemsep,nosep,
                        label=\rm(\alph*),
                        ref  =\alph*}
\setlist[tightenum,2]{noitemsep,nosep,
                         label=\rm(\roman*),
                         ref  =\roman*}
\setlist[tightenum,3]{noitemsep,nosep,
                         label=\rm(\roman*),
                         ref  =\roman*}

\crefname{p}{}{}
\creflabelformat{p}{#2(#1)#3}

\newenvironment{clmproof}{\begin{proof}[Proof of Claim:]}{\end{proof}}

\usepackage[longnamesfirst,numbers,sort&compress]{natbib}

\DeclareMathOperator{\len}{len}

\DeclareMathOperator{\dist}{dist}

\DeclareMathOperator{\polylog}{polylog}

\newcommand{\Oh}{\mathcal{O}}

\DeclarePairedDelimiter\set{\{}{\}}

\renewcommand{\mid}{:}  

\title{\MakeUppercase{{E}rdős–{P}ósa property of cycles that are far apart}}

\author{%
 Vida Dujmovi{\'c}\,\thanks{School of Computer Science and Electrical Engineering, University of Ottawa, Ottawa, Canada (\texttt{vida.dujmovic@uottawa.ca}). Research supported by NSERC and a University of Ottawa Research Chair.}
 \qquad
 Gwena\"el Joret\thanks{D\'epartement d'Informatique, Universit\'e libre de Bruxelles, Belgium ({\tt gwenael.joret@ulb.be}). G.\ Joret is supported by the Belgian National Fund for Scientific Research (FNRS).}
 \qquad
 Piotr Micek\thanks{Department of Theoretical Computer Science, Jagiellonian University, Kraków, Poland (\texttt{piotr.micek@uj.edu.pl}). Research supported by the National Science Center of Poland under grant UMO-2023/05/Y/ST6/00079 within the WEAVE-UNISONO program.}
 \qquad
 Pat Morin\thanks{School of Computer Science, Carleton University, Ottawa, Canada (\texttt{morin@scs.carleton.ca}). Research supported by NSERC.}%
}

\date{}

\begin{document}

\maketitle

\begin{abstract}
  We prove that there exist functions  $f,g:\mathbb{N}\to\mathbb{N}$ such that for all nonnegative integers $k$ and $d$,  for every graph $G$,  either $G$ contains $k$ cycles such that vertices of different cycles have distance greater than $d$ in $G$, or there exists a subset $X$ of vertices of $G$ with $|X|\leq f(k)$ such that  $G-B_G(X,g(d))$ is a forest, where $B_G(X,r)$ denotes the set of vertices of $G$ having distance at most $r$ from a vertex of $X$.
\end{abstract}

\section{Introduction}

In 1965, Erdős and Pósa~\cite{EP1965} showed that every graph $G$ either contains $k$ vertex-disjoint cycles or contains a set $X$ of $\Oh(k\log k)$ vertices such that $G-X$ has no cycles. This result has inspired an enormous amount of work on the so-called \emph{Erdős-Pósa property} of combinatorial objects. For example, using their Grid Minor Theorem, \citet{RS1986} proved the following generalization: For every planar graph $H$, there exists a function $f_H(k)$ such that every graph $G$ contains either $k$ vertex-disjoint subgraphs each having an $H$ minor, or a set $X$ of at most $f_H(k)$ vertices such that $G-X$ has no $H$ minor.

In a recent inspiring work, \citet{GP23} propose a new theme of research combining graph minors and coarse geometry. As part of this theme, they propose a number of conjectures including the Coarse Erdős-Pósa statement, see \cite[Conjecture 9.7]{GP23}. Our main theorem is the following natural variant of the Coarse Erdős-Pósa statement conjectured by Chudnovsky and Seymour.\footnote{The problem was discussed at the Barbados Graph Theory Workshop in March 2024 held at the Bellairs Research Institute of McGill University in Holetown.}

\begin{thm}\label{thm:main-in-intro}
  There exist functions $f,\, g:\mathbb{N}\to\mathbb{N}$ such that for all integers $k\ge 1$ and $d\ge 1$, for every graph $G$,  either $G$ contains $k$ cycles where the distance between each pair of cycles is greater than $d$, or  there exists a subset $X$ of vertices of $G$ with $|X|\leq f(k)$ such that  $G-B_G(X,g(d))$ is a forest.\footnote{Here, and throughout, $B_G(X,r)$ denotes the union of radius-$r$ balls in $G$ centered at the vertices in $X$.  Formal definitions appear in \cref{proof}.}
\end{thm}

We prove \cref{thm:main-in-intro} for $f(k)\in\Oh(k^{18}\polylog k)$ and $g(d)=19d$. 
The statement of \cref{thm:main-in-intro} was also posed as a conjecture by \citet{ahn.gollin:coarse}, who solved the cases $k=2$ and arbitrary $d$ as well as $d=1$ and arbitrary $k$.

\section{Tools}

Our proof makes use of two great tools. One is a key ingredient of Simonovits's~\cite{Simonovits67} proof of the Erdős and Pósa Theorem, which was published in 1967 just two years after the original paper. The other is a beautiful statement generalizing the Helly property of a family of subtrees in a fixed host tree, which is due to Gyárfás and Lehel and published in 1970.

For all positive integers $k$, define\footnote{Here and throughout, $\log x$ denotes the base-$2$ logarithm of $x$.}
\[
  \mathdefin{s(k)}:=
    \begin{cases}
      \lceil 4k(\log k + \log\log k +4)\rceil & \textrm{if $k\geq2$}\\
      2 & \textrm{if $k=1$.}
    \end{cases}
\]

\begin{thm}[\citet{Simonovits67}]\label{thm:simonovits}
  Let $k$ be a positive integer and let $G$ be a graph with all vertices of degree $2$ or $3$. If $G$ contains at least $s(k)$ vertices of degree $3$, then $G$ contains $k$ vertex-disjoint cycles.
\end{thm}

If a graph has no vertices then it is a \defin{null} graph, otherwise it is a \defin{non-null} graph.

\begin{restatable}[\citet{gyarfas.lehel:helly}]{thm}{hungarians}\label{thm:gyarfas-lehel-general}
   There exists a function $\ell^\star:\N\times\N\to\N$ such that the following is true. For every $k,c\in\N$ with $k,c\geq1$, for every forest $F$ and
   every set $\mathcal{A}$ of
   non-null subgraphs of $F$, each having at most $c$ components, either
   \begin{tightenum}
     \item there are $k$ pairwise vertex-disjoint members of $\mathcal{A}$; or\label[p]{item:hungarians-pack}
     \item there exists $X \subseteq V(F)$ with $|X|\leq \ell^\star(k,c)$ and such that $X\cap V(A)\neq\emptyset$ for each $A\in\mathcal{A}$.\label[p]{item:hungarians-hit}
   \end{tightenum}
\end{restatable}
Let $\ell^\star$ be the function given by~\cref{thm:gyarfas-lehel-general}. \citet{gyarfas.lehel:helly} give a complete proof of \cref{thm:gyarfas-lehel-general} when $F$ is a path and sketch the proof for the general case in which $F$ is a forest. For completeness, we give a proof in \cref{sec:hungarians}. We show that we can take $\ell^\star$ so that $\ell^\star(k,3)=\Oh(k^{18})$, which gives a bound of $f(k)\in\Oh(k^{18}\polylog k)$ in \cref{thm:main-in-intro}.

\section{Definitions and outline of the proof}
\label{sec:outline}

In this section we give an overview of the proof of~\cref{thm:main-in-intro}. 
We start with basic definitions. 
Let $G$ be a graph.  We use the notations $\mathdefin{V(G)}$ and $\mathdefin{E(G)}$ to denote the vertex set of $G$ and the edge set of $G$, respectively.  For a set $S$, $\mathdefin{G[S]}$ denotes the subgraph of $G$ induced by the vertices in $S\cap V(G)$, and $\mathdefin{G-S}:=G[V(G)\setminus S]$.  For a set $Z$ of edges of $G$, $G-Z$ denotes the subgraph of $G$ obtained by removing the edges in $Z$ from $G$.  For two subsets $A$ and $B$ of $V(G)$, we say that an edge $uv$ of $G$ with $u\in A$ and $v\in B$ is \defin{between} $A$ and $B$.

The \defin{length}, $\mathdefin{\len(P)}$, of a path $P$ is the number of edges in $P$. 
The \defin{distance} between two vertices $x$ and $y$ of $G$, denoted by $\mathdefin{\dist_G(x,y)}$, is the length of a shortest path in $G$ with endpoints $x$ and $y$ if there is one, or $\infty$ if there is none. For nonempty subsets $X$ and $Y$ of $V(G)$, define $\mathdefin{\dist_G(X,Y)}:=\min\{\dist_G(x,y):(x,y)\in X\times Y\}$. For an integer $r\ge 0$ and a vertex $x$ of $G$, the \defin{ball of radius $r$ around $x$} is $\mathdefin{B_G(x,r)}:=\{y\in V(G):\dist_G(x,y)\le r\}$.  For a subset $X$ of $V(G)$, $\mathdefin{B_G(X,r)}:=\bigcup_{x\in X}B_G(x,r)$.

Let $G$ be a graph and let $d$ be a nonnegative integer. A set $\mathcal{C}$ of cycles in $G$ is a \defin{$d$-packing} if $\dist_G(V(C),V(C'))> d$ for every two distinct $C,C'\in\mathcal{C}$.

A graph $G$ is \defin{unicyclic} if $G$ contains at most one cycle. 
(In particular, if a unicyclic graph contains a cycle then this cycle must be chordless.)
Let $r$ be a nonnegative integer and let $G$ be a graph. A cycle $C$ in $G$ is \defin{$r$-unicyclic in $G$} if $G[B_G(V(C),r)]$ is unicyclic. (Note that every induced cycle in $G$ is $0$-unicyclic.)

We continue with an overview of the proof of~\cref{thm:main-in-intro}. 
Let $k\geq1$, $d\geq1$, and let $G$ be a graph. 
Suppose that $G$ does not contain a $d$-packing of $k$ cycles (otherwise, we are done). 
The goal is to find a small number of balls (bounded in $k$) of small radii (bounded in $d$) that hit all the cycles in $G$.

First, we take a  maximal $2d$-packing $\mathcal{C}=\set{C_1,\ldots,C_p}$ of cycles that are $d$-unicyclic in $G$.  
Since $G$ has no $d$-packing of $k$ cycles, $p<k$. 
Next, in a simple iterative process, we build a $d$-packing 
$\mathcal{D}=\set{D_1,\ldots,D_q}$ of 'short' cycles in $G$, of length at most $6d+2$. 
Again, we have $q<k$. 
The key property of these two collections is that every cycle in $G$ contains a vertex 
in $B_G(\bigcup_{i\in[p]} V(C_i),4d) \cup B_G(\bigcup_{i\in[q]} V(D_i),4d)$. 
This is a consequence of~\cref{short_or_unicycle_nearby}, which states that for every cycle $C$ in $G$, either the ball of radius $2d$ around $C$ contains a $d$-unicyclic cycle or 
the ball of radius $3d$ around $C$ contains a short cycle. 

We choose one vertex from each cycle in $\mathcal{D}$ and let $X_0$ be the set of chosen vertices. 
Note that since all cycles in $\mathcal{D}$ are short, we have $B_G(\bigcup_{i\in[q]}V(D_i)) \subseteq B_G(X_0,7d+1)$.

Next, we aim to hit all the cycles in $G$ contained in $B_G(V(C_i),6d)$, for each $i\in[p]$. 
Fix $i\in[p]$ and fix a $C_i$-rooted spanning BFS-unicycle $U_i$ of $G_i:=G[B_G(V(C_i),6d)]$ (see next section for the definition). 
Each edge $e$ in $E(G_i)\setminus E(U_i)$ together with two paths in $U_i$ from the endpoints of $e$ to $V(C_i)$ induces a subgraph $P_e$ of $G$ that is a path or a 'lollipop' in $G_i$.  
One can prove that either there are many such subgraphs $P_e$ that are $d$-apart in $G$ or there is a subset of vertices $Z_i$ such that $B_G(Z_i,13d)$ hits every edge $e$ in $E(G_i)\setminus E(U_i)$. In the former case, 
this subfamily of paths $\set{P_e}$ together with cycles in $\mathcal{C}$ form a 'frame', a subgraph of $G$ where all vertices have degree $2$ or $3$, and with many vertices of degree $3$. One can then apply Simonovits's Theorem, see~\cref{thm:simonovits}, 
which gives a family of $k$ vertex-disjoint cycles in the frame. 
This packing of cycles in the frame translates into a $d$-packing of $k$ cycles in $G$, a contradiction.
Thus, we obtain the second outcome, that is, a subset $Z_i$ of vertices in $G$ such that $B_G(Z_i,13d)$ contains an endpoint of every edge $e$ in $E(G_i)\setminus E(U)$. 
This argument is encapsulated in~\cref{grow_unicycle}. 
We then let $X_1:=\bigcup_{i\in[p]} Z_i$. 

With a similar argument, 
which involves another application of Simonovits Theorem, and 
which is enclosed in~\cref{double_unicycle}, we construct a set $X_2\subseteq V(G)$  of bounded size such that all vertices of $G$ that lie in the intersection of at least two sets from $\set{B_G(V(C_i),6d)\mid i\in[p]}$ are contained in $B_G(X_2,13d)$. 

At this point, we are in the following situation: We defined bounded-size sets $X_0, X_1, X_2$, and we let $\widehat{Y}$ be the union of the balls centered on the vertices in these sets (with the appropriate bounded radius). 
After the removal of $\widehat{Y}$, the graph $G$ has no cycle that is at distance more than $4d$ from all cycles in $\set{C_1,\ldots,C_p}$. Thus, the graph 
\[
F_0 := \textstyle{G-\left(\widehat{Y}\cup \bigcup_{i\in[p]} B_G(V(C_i),4d)\right)}
\]
is a forest. 
Similarly, for each $i\in[p]$, we made sure that every cycle in the $6d$-ball around $C_i$ is hit by $\widehat{Y}$, thus 
\[
F_i := G[B_G(V(C_i),6d)]-\widehat{Y}
\]    
is a forest as well. 
Now, every cycle in $G$ that avoids $\widehat{Y}$ will necessarily intersect at least two of the forests $F_0, F_1, \dots, F_p$ (possibly many of them). 
Moreover, every such cycle $C$ must intersect $F_0$ (which we think of as the "forest outside"), thanks to the properties of $\widehat{Y}$, namely, $\widehat{Y}$ contains all vertices of $G$ that are contained in the intersection of two $6d$-balls around two cycles in $\set{C_1,\ldots,C_p}$. In other words, if our cycle $C$ enters the $6d$-ball around some cycle $C_i$, then it has to exit it via $F_0$. 
This naturally leads us to consider the 'outside segments' of $C$, which are (roughly) the parts of $C$ contained in $F_0$, and the 'inside segments' of $C$ that belong to some $F_i$ with $i\in [p]$. 

The final step of the proof, which is perhaps the most technical one, is to consider the $d$-balls around all the segments defined by all cycles of $G$ that avoid $\widehat{Y}$, and use them in an application of \cref{thm:gyarfas-lehel-general}---a powerful generalization of the Helly property of subtrees of a tree---using crucially that the number of 'dimensions' in our application is bounded: one for each of the forests $F_0, F_1, \dots, F_p$. 
The latter result either gives a bounded number of vertices hitting the $d$-balls of all these segments, which we use as centers of balls to hit the remaining cycles in $G$. 
Or it gives a large number of segments with pairwise disjoint $d$-balls. 
In the latter case, we use these to find $k$ cycles pairwise at distance more than $d$ using an application of Simonovits Theorem again. 

We remark that the second part of the above outline, once $\widehat{Y}$ has been built, is a simplified view of the actual proof that does not mention various subtleties but it is good enough to get a first intuitive picture of the proof.

\section{The Proof}
\label{proof}

\begin{lem}\label{short_or_unicycle_nearby}
  Let $r$ be a nonnegative integer, let $G$ be a graph, and let $C$ be a cycle in $G$. Then either
  \begin{tightenum}
    \item $G$ has an $r$-unicyclic cycle $C'$ with $V(C')\subseteq B_G(V(C),2r)$,  or\label{short_or_unicycle_nearby:unicyclic}
    \item $G$ has a cycle $C'$ of length at most $6r+2$ with $V(C')\subseteq B_G(V(C),3r)$.\label{short_or_unicycle_nearby:short}
  \end{tightenum}
\end{lem}

\begin{proof}
  Let $C_0:=C$.
  We construct inductively a sequence of pairs $(C_i,Q_i)_{i\geq0}$ such that
  \begin{enumerate*}[label=\rm(\arabic*)]
    \item $C_i$ is a cycle in $G$;
    \item $Q_i\subseteq C_i\cap C$, $Q_i$ is a path and contains at least one edge; and
    \item $C_i$ contains at most $4r+1$ edges not in $Q_i$.
  \end{enumerate*}
  Note that the first two conditions imply that there is a path $P_i\subseteq C_i$  such that $P_i$ and $Q_i$ are edge-disjoint and $C_i=P_i\cup Q_i$. Thus, $P_i$ is a path of length at most $4r+1$ with both endpoints in $C$ which implies that all vertices of $P_i$ are at distance at most $2r$ from $C$ in $G$. In particular, $V(C_i)\subseteq B_G(V(C),2r)$.

  Let $i\geq0$ and suppose that we already have defined $(C_i,Q_i)$. If $B_G(V(C_i),r)$ is unicyclic, then $C_i$ witnesses~\eqref{short_or_unicycle_nearby:unicyclic}, and we are done.

  Now suppose that $G[B_G(V(C_i),r)]$ contains a cycle $D$ distinct from $C_i$.  If $C_i$ has a chord $xy$, then we define a path $P:=xy$.  Otherwise, we construct a path $P$ as follows. Since $C_i$ has no chord, $D$ contains a vertex not in $C_i$.  Let $u$ be a vertex of $D$ that is at maximum distance from $C_i$ in $G$ among all vertices in $D$. Let $P(u)$ be a shortest path in $G$ from $V(C_i)$ to $u$. Since $u$ is adjacent to two vertices in $D$ there is a neighbor $v$ of $u$ in $D$ that is not in $P(u)$. Let $P(v)$ be a shortest path from $V(C_i)$ to $v$ in $G$. Note that both $P(u)$ and $P(v)$ have length at most $r$. By the choice of $u$, the path $P(v)$ does not contain $u$. See Figure~\ref{fig:from-Ci-to-Ci+1}. If $P(u)\cup P(v)\cup\{uv\}$ contains a cycle, say $E$, then $E$ has length at most $2r+1$ and $V(E)\subseteq B_G(V(C_i),r)\subseteq B_G(V(C),3r)$, so $E$ witnesses~\eqref{short_or_unicycle_nearby:short}. Thus, we assume that $P:=P(u)\cup P(v)\cup\{uv\}$ contains no cycle, so $P$ must be a path.
  \begin{figure}[t]
  \begin{center}
  \includegraphics{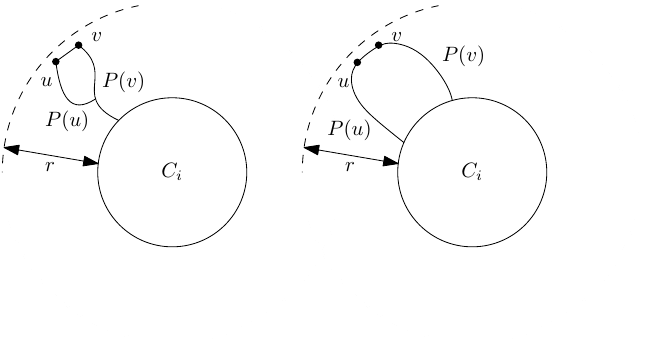}
  \end{center}
  \caption{Proof of~\cref{short_or_unicycle_nearby}: on the left $P(u)\cup P(v)\cup\set{uv}$ contains a cycle; on the right $P(u)\cup P(v)\cup\set{uv}$ is a path with both ends in $V(C_i)$.}
  \label{fig:from-Ci-to-Ci+1}
\end{figure}

  In each case, $P$ is a path with both endpoints in $C_i$ and no internal vertices in $C_i$ and $\len(P)\le 2r+1$.  Now there are three possibilities, depending on where the endpoints of $P$ are located:
  \begin{enumerate}[noitemsep,nosep]
    \item Both endpoints of $P$ are in $V(P_{i})$. In this case, $P\cup P_{i}$ contains a cycle of length at most $(2r+1)+(4r+1)=6r+2$ and this cycle is contained in $B_G(V(C_i),r)\subseteq B_G(V(C),3r)$, so it satisfies~\eqref{short_or_unicycle_nearby:short}.
    \item Both endpoints of $P$ are in $V(Q_{i})\setminus V(P_{i})$. In this case, we take $C_{i+1}$ to be the cycle in $Q_{i}\cup P$ and we take $Q_{i+1}:= Q_i\cap C_{i+1}$.  This works because $P$ has two distinct endpoints in $Q_{i+1}$ and therefore $Q_{i+1}$ contains at least one edge, $Q_{i+1}\subseteq Q_i\subseteq C$, and $\len(P_{i+1})\leq 2r+1$ since $P_{i+1}=P$ in this case. Furthermore, $\len(Q_{i+1}) < \len(Q_{i})$ because $C_{i+1}$ does not contain either endpoint of $Q_{i}$.
    \item Exactly one endpoint of $P$ is in $V(P_{i})$. In this case, $C_{i}\cup P$ has two cycles that each contain $P$. Each edge of $P_{i}$ belongs to exactly one of these two cycles. Therefore one of these cycles uses at most $\lfloor\frac{4r+1}{2}\rfloor=2r$ edges of $P_{i}$. We take $C_{i+1}$ to be this cycle and define $Q_{i+1}=C_{i+1}\cap Q_i$. 
    Note that $\len(Q_{i+1}) < \len(Q_{i})$ because $C_{i+1}$ does not contain one of the endpoints of $Q_{i}$.
  \end{enumerate}
  This process eventually produces the desired cycle $C'$ from the statement since, at each step in the process $\len(Q_i)$ decreases.
\end{proof}

The following definition, which is a unicyclic analogue to a breadth-first spanning tree, plays a central role in our proofs. Given a cycle $C$ in a connected graph $G$, a \defin{$C$-rooted spanning BFS-unicycle} of $G$ is a spanning connected unicyclic subgraph $U$ of $G$ with $E(U)\supseteq E(C)$ and such that $\dist_U(v,V(C))=\dist_G(v,V(C))$ for each $v\in V(G)$.  For each vertex $v$ in such a subgraph $U$, $v$ is a \defin{$U$-descendant} of every vertex in the unique shortest path in $U$ from $v$ to $V(C)$.

\begin{lem}\label{all_the_ys}
  Let $d,r,k$ be integers such that $1\leq d \leq r$ and $k\geq1$. Let $G$ be a graph and let $\mathcal{C}$ be a $2d$-packing of $d$-unicyclic cycles in $G$. Then either:
  \begin{tightenum}
    \item\label[p]{atys:packing} $G$ has a $d$-packing of $k$ cycles; or

    \item there exist $X,Y\subseteq V(G)$ such that
    \begin{tightenum}
      \item\label[p]{atys_unicycles} for each $C\in\mathcal{C}$ there exists a $C$-rooted spanning BFS-unicycle $U$ of $G[B_G(V(C),r)]$ such that $Y$ contains at least one endpoint of each edge in $G[B_G(V(C),r)]- E(U)$;
      \item\label[p]{atys_intersections} for each $\set{C,C'}\in\binom{\mathcal{C}}{2}$,  $B_{G}(V(C),r)\cap B_{G}(V(C'),r)\subseteq Y$; and 
      \item\label[p]{atys_balls} $|X|< 2k^2+(\binom{k}{2}+k)\cdot s(k)$ and $Y\subseteq B_G(X,2r+d)$.
    \end{tightenum}
  \end{tightenum}
\end{lem}

\Cref{all_the_ys} is an easy consequence of the next two technical lemmas, \cref{grow_unicycle,double_unicycle}, that address \cref{atys_unicycles} and \cref{atys_intersections} somewhat separately.  Thus, we defer the proof of \cref{all_the_ys} until after the proofs of \cref{grow_unicycle,double_unicycle}.

\begin{lem}\label{grow_unicycle}
  Let $d,r,k$ be integers such that $r\geq d \geq 1$ and $k\geq1$. Let $G$ be a graph and let $C$ be a $d$-unicyclic cycle in $G$. Let $U$ be a $C$-rooted spanning BFS-unicycle of $G_r:=G[B_G(V(C),r)]$.
  Then either:
  \begin{tightenum}
    \item\label[p]{grow_unicycle:item:packing} $G$ has a $d$-packing of $k$ cycles; or
    \item\label[p]{grow_unicycle:item:hitting} there exists $Y\subseteq V(G)$ that contains at least one endpoint of each edge in $E(G_r)\setminus E(U)$ and  there exists $X\subseteq V(G)$ with $|X|<2k+s(k)$ such that      $B_G(X,2r+d)$ contains all $U$-descendants of vertices in $Y$.
  \end{tightenum}
\end{lem}
  \begin{figure}[t]
  \begin{center}
  \includegraphics{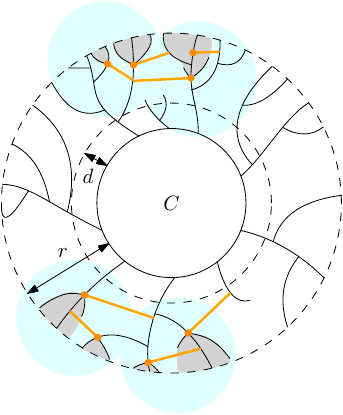}
  \end{center}
  \caption{Illustration of outcome \cref{grow_unicycle:item:hitting} of~\cref{grow_unicycle}: 
  $C$ is a $d$-unicyclic cycle in $G$. 
  $U$ is a $C$-rooted spanning BFS-unicycle of $G_r=G[B_G(V(C),r)]$. 
  The edges of $E(G_r)\setminus E(U)$ are highlighted, and one of endpoint of each such edge is highlighted. The set of these highlighted endpoints is $Y$. 
  Finally, $X$ is the set of centers of a small number of balls (four in this figure) whose union contains all $U$-descendants of vertices in $Y$.}
  \label{fig:killing-edges-outside-the-unicycle}
\end{figure}

\begin{proof}
  Let $e_0$ be an arbitrary edge of $C$ and observe that $U-\{e_0\}$ is a tree.  For each edge $e\in E(G_{r})\setminus E(U)$,  define $C_e$ to be the unique cycle in $U\cup \set{e}$ that does not contain $e_0$ and define $P_{e}$ to be the path or cycle formed by the edges in $E(C_{e})\setminus E(C)$. If $C_e$ has no edges in $C$ then we say that $e$ is \defin{$C$-null}. Otherwise, we say that $e$ is \defin{$C$-nonnull}.

  Consider the auxiliary graph $H$ with vertex set $V(H):= E(G_{r})\setminus E(U)$ in which two distinct elements $e$ and $e'$ are adjacent in $H$ if $\dist_G(V(P_{e}),V(P_{e'})) \le d$.  Let $I$ be a maximal independent set in $H$. We split the proof into two cases:
  \begin{enumerate*}[label=\rm(\arabic*)]
    \item $|I|\ge k+\frac{1}{2}s(k)$, in which case we show that \cref{grow_unicycle:item:packing} holds; or
    \item $|I|< k+\frac{1}{2}s(k)$, in which case we show that \cref{grow_unicycle:item:hitting} holds.
  \end{enumerate*}

  We first consider the case where $|I|\ge k+\frac{1}{2}s(k)$.  Let $J$ be the set of $C$-null edges in $I$. We now show that $\set{C_e \mid e\in J}$ is a $d$-packing of cycles in $G$.  Consider two distinct $e,e'\in J$.  We must show that $\dist_G(V(C_e),V(C_{e'}))>d$. Since $e$ and $e'$ are $C$-null, $P_e=C_e$ and $P_{e'}=C_{e'}$. Since $e$ and $e'$ are both in $I$, $\dist_G(V(C_e),V(C_{e'}))=\dist_G(V(P_e),V(P_{e'}))>d$. Thus, $\set{C_e \mid e\in J}$ is indeed a $d$-packing of $G$.  If $|J|\ge k$ then this establishes \eqref{grow_unicycle:item:packing}.

  Otherwise $|J|<k$, so $|I\setminus J|\ge \tfrac{1}{2}s(k)$. Let $G':=C\cup\bigcup_{e\in I\setminus J} P_e$. Since $\dist_G(V(P_e),V(P_{e'}))>d\ge0$ for all distinct $e,e'\in J$, $G'$ contains only vertices of degree $2$ or $3$, and the degree-$3$ vertices of $G'$ are the endpoints of paths in $\set{P_e \mid e\in I\setminus J}$. Therefore, $G'$ contains $2|I\setminus J|\geq s(k)$ vertices of degree $3$.

  By \cref{thm:simonovits}, $G'$ contains a set $\mathcal{D}$ of $k$ pairwise vertex-disjoint cycles. We now show that $\mathcal{D}$ is a $d$-packing of cycles in $G$.  
  Let $D$ and $D'$ be two distinct cycles in $\mathcal{D}$. Let $v$ and $v'$ be vertices of $D$ and $D'$, respectively, such that $\dist_G(v,v')=\dist_G(V(D),V(D'))$.  We must show that $\dist_G(v,v')>d$.

  If $v\in V(P_e)$ and $v'\in V(P_{e'})$ for some  $e,e' \in I$ such that $P_e\subseteq D$ and $P_{e'}\subseteq D'$ then, since $D$ and $D'$ are vertex-disjoint, $e\neq e'$. Therefore $\dist_G(v,v')\geq  \dist_G(V(P_e),V(P_{e'}))>d$, as desired.  Thus, without loss of generality we may assume that $v\in V(C)$ and that $v\notin V(P_e)$ for any $e\in I\setminus J$ with $P_e\subseteq D$.

  Let $P$ be a shortest path in $G$ between $v$ and $v'$.  Then $P$ is a shortest path in $G$ from $V(D)$ to $V(D')$, so $E(P)\cap E(D)=\emptyset$ and $E(P)\cap E(D')=\emptyset$.  If the first edge of $P$ (incident to $v$) is an edge of $C$, then this edge and the two edges of $D$ incident to $v$ are contained in $G'$. Therefore $v$ is of degree $3$ in $G'$ and $v$ is the endpoint of $P_e$ for some $e\in I$ such that $P_e\subseteq D$.  This contradicts the definition of $v$. Thus, we assume that the first edge of $P$ is not an edge of $C$.

  If $P$ contains a vertex not in $B_G(V(C),d)$ then $\dist_G(v,v')=\len(P)>d$, as required.  Thus we assume that $V(P)\subseteq B_G(V(C),d)$. In particular, since $C$ is $d$-unicyclic in $G$ and the first edge of $P$ is not an edge of $C$, $V(P)\cap V(C)=\set{v}$.

  Since $V(P)\cap V(C)=\set{v}$,  $v'\notin V(C)$.  Therefore $v'$ is an internal vertex of $V(P_{e'})$ for some $e'\in I\setminus J$. 
  Since every internal vertex of $P_{e'}$ has degree $2$ in $G'$, this implies that  $P_{e'}\subseteq D'$.    
  Since $C$ is $d$-unicyclic in $G$, every path in $G[B_G(V(C),d)]$ from $V(C)$ to $v'$ contains an endpoint of $P_{e'}$.  In particular, $P$ contains an endpoint of $P_{e'}$.  Since $V(P)\cap V(C)=\{v\}$, this implies that $v$ is an endpoint of $P_{e'}$.  Then $v$ is a vertex of $P_{e'}\subseteq D'$, which contradicts the fact that $D$ and $D'$ are vertex-disjoint.

  It remains to consider the case when $|I| < k+\frac{1}{2}s(k)$.  We now define $X\subseteq V(G)$ and $Y\subseteq V(G)$ and show that these sets satisfy \cref{grow_unicycle:item:hitting}.

  Since $I$ is a maximal independent set in $H$, it is a dominating set in $H$: Every vertex of $H$ is either in $I$ or adjacent to a vertex in $I$.  In $G$ this corresponds to the fact that, for every edge of $e=uv\in E(G_r)\setminus E(U)$ there exists $e'=u'v'\in I$ such that $\dist_G(V(P_{e}),V(P_{e'}))\le d$.  (Possibly $e=e'$ which can occur when $e\in I$.)

  Define $X:=\bigcup_{uv\in I}\{u,v\}$.  Then $|X|\le 2|I| < 2k+s(k)$.  To define $Y$ we process each edge $e:=uv\in E(G_r)\setminus E(U)$ as follows: Let $e':=u'v'\in I$ be such that $\dist_G(V(P_{e}),V(P_{e'}))\le d$.  We choose an endpoint of $e$ to include in $Y$ and show that $B_G(\{u',v'\}, 2r+d)$ contains this endpoint. Since $\dist_G(u',V(C))\leq r$ and $\dist_G(v',V(C))\leq r$, we have $V(P_{e'}) \subseteq B_G(\set{u',v'},r)$. Since $\dist_G(V(P_e),V(P_{e'}))\le d$,  $B_G(\set{u',v'},r+d)$ contains a vertex of $P_e$. In particular, $B_G(\set{u',v'},r+d)$ contains a vertex on the shortest path $P_u$ in $U$ from $u$ to $V(C)$ or $B_G(\set{u',v'},r+d)$ contains a vertex on the shortest path $P_v$ in $U$ from $v$ to $V(C)$.

  Without loss of generality, suppose $B_G(\set{u',v'},r+d)$ contains a vertex $x$ of $P_u$.  Then $B_G(x,r)$ contains all $U$-descendants of $u$.  Therefore, $B_{G}(\set{u',v'},2r+d)$ contains all $U$-descendants of $u$.  We add $u$ to $Y$. Processing each edge $uv\in E(G_r)\setminus E(U)$ this way produces a set $Y$ so that $X$ and $Y$ satisfy \cref{grow_unicycle:item:hitting}.
\end{proof}

\begin{lem}\label{double_unicycle}
  Let $r,d,k$ be integers with $1 \leq d\leq r$ and $k\geq1$. 
  Let $G$ be a graph, let $C_1$ and $C_2$ be $d$-unicyclic cycles in $G$ such that $\dist_G(V(C_1),V(C_2))>2d$.  For each $i\in[2]$ let $G_{i,r}:=G[B_G(V(C_i),r)]$, let $U_i$ be a $C_i$-rooted spanning BFS-unicycle of $G_{i,r}$ and let $Y_i\subseteq V(G)$ be such that each edge in $E(G_{i,r})\setminus E(U_i)$ has an endpoint in $Y_i$.   Then either:
  \begin{tightenum}
    \item\label[p]{double_unicycle:packing} $G$ has a $d$-packing of $k$ cycles; or
    \item\label[p]{double_unicycle:hitting} there exist $Y\subseteq V(G)$ and $X\subseteq V(G)$ with $|X|< s(k)$ such that
    \begin{tightenum}
      \item\label[p]{y_covers_everything} $Y$ contains every vertex in $V(G_{1,r})\cap V(G_{2,r})$; 
      and
      \item for each vertex $y\in Y$ either \label[p]{something_hits_y}
      \begin{enumerate}[nosep,noitemsep,label=\rm(\alph*),ref=\alph*,start=24]
        \item\label[p]{du_yi_hits_y}  $y$ is a $U_i$-descendant of a vertex in $Y_i$ for some $i\in[2]$; or
        \item\label[p]{du_ball_hits_y} $y\in B_G(X,2r+d)$.
      \end{enumerate}
    \end{tightenum}
  \end{tightenum}
\end{lem}

\begin{proof}
  For each $i\in[2]$ and each $v\in V(U_i)$, let $P_{i,v}$ be the shortest path from $v$ to $V(C_i)$ in $U_i$. Let $Y:=B_G(V(C_1),r)\cap B_G(V(C_2),r)$. 
  For each $x$ in 
  $Y$, 
  define $Q_x:=P_{1,x}\cup P_{2,x}$ and define $P_x$ to be a shortest path from $V(C_1)$ to $V(C_2)$ in $Q_x$.

  Let $x$ be an element of $Y$.   
  Since $\dist(V(C_1),V(C_2))>d\ge 0$, the cycles $C_1$ and $C_2$ are vertex-disjoint.  The path $P_x$ is a non-trivial path between $V(C_1)$ and $V(C_2)$ of length at least $1$ and at most $\len(P_{1,x})+\len(P_{2,x})\le 2r$.  ($P_x$ does not necessarily contain the 
  vertex $x$.) We say that 
  $x\in Y$ 
  is \defin{uninteresting} if $P_x$ contains a $U_i$-descendant of some vertex in $Y_i$ for some $i\in[2]$,  otherwise $x$ is \defin{interesting}.

  Let $H$ be an auxiliary graph whose vertex set is the set of interesting elements in  
  $Y$ 
  and in which two distinct vertices $x$ and $x'$ are adjacent if and only if $\dist_G(V(P_x),V(P_{x'}))\leq d$. Let $I$ be a maximal independent set in $H$. We split the proof into two cases.  If $|I|\geq \frac{1}{2}s(k)$ we will show that outcome \cref{double_unicycle:packing} holds.  Otherwise, we show that outcome \cref{double_unicycle:hitting} holds.

  First consider the case in which $|I|\ge\tfrac{1}{2}s(k)$. Let $G':=C_1\cup C_2\cup\bigcup_{x\in I}P_x$. For distinct $x,x'\in I$, $xx'$ is not an edge of $H$, so $\dist_G(V(P_x),V(P_{x'}))>d\ge 1$. This together with the fact that $C_1$ is vertex-disjoint from $C_2$ gives that $G'$ contains only vertices of degree $2$ and degree $3$. Recall that for each 
  $x\in Y$, 
  the path $P_x$ contains two different endpoints. Therefore, the degree-$3$ vertices of $G'$ are exactly the endpoints of paths in $\set{P_x\mid x\in I}$. This implies that $G'$ contains $2|I|\geq s(k)$ vertices of degree $3$. By \cref{thm:simonovits}, $G'$ contains a set $\mathcal{D}$ of $k$ pairwise vertex-disjoint cycles.

  We claim that $\mathcal{D}$ is a $d$-packing of cycles in $G$. 
  Let $D$ and $D'$ be two distinct cycles in $\mathcal{D}$ and let $v\in V(D)$ and $v'\in V(D')$ be such that $\dist_G(v,v')=\dist_G(V(D),V(D'))$.  We must show that $\dist_G(v,v')>d$.

  If $v\in V(P_x)$ and $v'\in V(P_{x'})$ for some  $x,x' \in I$ such that $P_x\subseteq D$ and $P_{x'}\subseteq D'$ then, since $D$ and $D'$ are vertex-disjoint, $x\neq x'$, and therefore $\dist_G(v,v'))\ge  \dist_G(V(P_x),V(P_{x'}))>d$, as desired.  Therefore, without loss of generality we may assume that $v\in V(C_1)$ and that $v\notin V(P_x)$ for any $x\in I$ with $P_x\subseteq D$.

  Let $P$ be a shortest path in $G$ between $v$ and $v'$.  Since $P$ is a shortest path in $G$ from $V(D)$ to $V(D')$, $E(P)\cap E(D)=\emptyset$ and $E(P)\cap E(D')=\emptyset$.  If the first edge $e$ of $P$ (incident to $v$) is an edge of $C_1$, then $e$ and the two edges of $D$ incident to $v$ are contained in $G'$. Therefore $v$ has degree $3$ in $G'$ and $v$ is the endpoint of $P_y$ for some $y\in I$ such that $P_y\subseteq D$.  This contradicts the definition of $v$. Thus, we assume that the first edge of $P$ (incident to $v$) is not in $C_1$.

  If $P$ contains a vertex not in $B_G(V(C_1),d)$ then $\dist_G(v,v')=\len(P)>d$, as required.  Thus we assume that $V(P)\subseteq B_G(V(C_1),d)$. Since $C_1$ is $d$-unicyclic in $G$, this implies that $P\subseteq U_1$ with $\len(P)\leq d$ and $V(P)\cap V(C_1)=\set{v}$. Since $v'$ is an endpoint of $P$, this implies that $v'\notin V(C_1)$.  If $v'\in V(C_2)$ then $\dist_G(v,v')\ge \dist_G(V(C_1),V(C_2))>2d\geq d$, as required.

  The only remaining possibility is that $v'$ is an internal vertex of $P_{x'}$ for some $x'\in I$.  Since every internal vertex of $P_{x'}$ has degree $2$ in $G'$, this implies that  $P_{x'}\subseteq D'$. Let $w'$ be the endpoint of $P_{x'}$ in $C_1$. Since $x'\in I\subseteq V(H)$, $x'$ is interesting. Therefore $P_{x'}\subseteq P_{1,x'}\cup P_{2,x'}$ contains no vertex in $Y_1\cup Y_2$.

  Consider the maximal subpath $Q'$ of $P_{x'}$ that contains $w'$ and whose vertices are contained in $B_G(V(C_1),r)$.  We distinguish between two cases, depending on whether $v'\in V(Q')$ or not.

  First suppose that $v'\in V(Q')$. Let $Q$ be the subpath of $Q'$ from $w'$ to $v'$. Since $Q\subseteq Q'\subseteq P_{x'}$ contains  no vertex in $Y_1$ and $V(Q')\subseteq B_G(V(C_1),r)$, we have that $Q\subseteq U_1$. Since $U_1$ is a $C_1$-rooted spanning BFS-unicycle of $G_{1,r}$, $\dist_G(w',v')=\dist_{U_1}(w',v')=\len(Q)$. Observe that $Q$ is a shortest path in $G$ between $V(C_1)$ and $v'$, and thus $\len(Q) \leq \len(P) \leq d$. It follows that $Q\subseteq G[B_G(V(C_1),d)]$. Since $C_1$ is $d$-unicyclic, all paths from $V(C_1)$ to $v'$ in $G[B_G(V(C_1),d)]$ contain $w'$. In particular, $P\subseteq G[B_G(V(C_1),d)]$ is a path from $v\in V(C_1)$ to $v'$, so $w'\in V(P)$.  Since $V(P)\cap V(C_1)=\{v\}$, this implies that $v=w'$. However $v=w' \in V(P_{x'})\subseteq V(D')$. This implies that $v$ is a vertex of both $D$ and $D'$, which contradicts the fact that $D$ and $D'$ are vertex-disjoint.

  Next we consider the case when $v'\notin V(Q')$. Let $q$ be the last vertex of $Q'$ and let $q'$ be the neighbour of $q$ in $P_{x'}$ that is not in $Q'$.  Then
  \begin{align*}
  \len(P_{x'})
    & \ge \dist_G(w',q') + \dist_G(q',v') + \dist_G(v',V(C_2)) \\
    & \ge \dist_G(V(C_1),q') + (\dist_G(V(C_1),q') - \dist_G(V(C_1),v')) \\
    & \quad {} + (\dist_G(V(C_1),V(C_2))-\dist_G(v',V(C_1)) \\
    & = 2\dist_G(V(C_1),q') - 2\dist_G(V(C_1),v') + \dist_G(V(C_1),V(C_2)) \\
    & \geq 2(r+1) - 2d + (2d+1) \geq 2r+3. 
  \end{align*}
  (The second inequality follows from $\dist_G(V(C_1),q')\leq \dist_G(V(C_1),v')+\dist_G(v',q')$ and $\dist_G(V(C_1),V(C_2))\le \dist_G(V(C_1),v')+\dist_G(v',V(C_2))$.) This contradicts the fact that 
  $\len(P_{x'})\le \len(P_{1,x'})+\len(P_{2,x'})\le 2r$.  
  This completes the proof (assuming $|I|\ge\tfrac{1}{2}s(k)$) that outcome \cref{double_unicycle:packing} holds.

  What remains is to consider the case where $|I|<\frac{1}{2}s(k)$.  Let
  \[
    \textstyle X:=\bigcup_{y\in I}\left( V(P_y)\cap(V(C_1)\cup V(C_2))\right), 
  \]
  be the set of degree-$3$ vertices in $G'$. Note that $|X| = 2\cdot|I| < s(k)$. 
  We show that $X$ and $Y$ satisfy \cref{double_unicycle:hitting}.


  Consider some vertex $y\in Y$.  
  If $y$ is a $U_i$-descendant of some vertex in $Y_i$ for some $i\in[2]$, then $y$ satisfies \cref{du_yi_hits_y}. 
  Therefore, we assume that this is not the case, which implies that $y$ is interesting, so $y\in V(H)$. Since $I$ is a maximal independent set in $H$, the set $I$ is also a dominating set. Therefore, either $y\in I$ or $y$ has a neighbor in $I$.  In $G$, this translates into the statement: there exists $y'\in I$ such that $\dist_G(V(P_y),V(P_{y'}))\le d$.
  (This includes the case $y=y'$, which applies when $y\in I$.)

  Let $y'\in I$ be such that $\dist_G(V(P_{y'}),V(P_y))\le d$.  Then $X$ contains the endpoints $u_1'$ and $u_2'$ of $P_{y'}$.  Then $V(P_{y'})\subseteq V(P_{1,y'}\cup P_{2,y'})\subseteq B_G(\{u_1',u_2'\},r)$.  Since $\dist_G(V(P_{y'}),V(P_y))\le d$, we have that $B_G(\{u_1',u_2'\},r+d)$ contains some vertex in $V(P_y)\subseteq V(P_{1,y})\cup V(P_{2,y})$, and thus $B_G(\{u_1',u_2'\},2r+d)$ contains $y$, as desired. 
  

  We have just shown that $X$ and $Y$ satisfy \cref{something_hits_y}.  Furthermore, $Y=V(G_{1,r})\cap V(G_{2,r})$. 
  Thus, the sets $X$ and $Y$ satisfy the conditions of outcome \cref{double_unicycle:hitting}.
\end{proof}

\begin{proof}[Proof of \cref{all_the_ys}]
  Let $\mathcal{C}$ be a $2d$-packing of $d$-unicyclic cycles in $G$. If $|\mathcal{C}|\geq k$, then~\cref{grow_unicycle:item:packing} holds. Thus, let $\mathcal{C}:=\{C_1,\ldots,C_p\}$ and we assume that $p<k$.

  For each $i\in[p]$, let $U_i$ be a spanning BFS-unicycle of $G[B_G(V(C_i),r)]$. 
  For each $i\in[p]$, we apply~\cref{grow_unicycle} to $C_i$ and $U_i$.
  If any of these applications of \cref{grow_unicycle} produces a $d$-packing of $k$ cycles in $G$, then $G$ satisfies (a) and there is nothing to prove.  
  Thus, for each $i\in[p]$ we obtain $Y_i\subseteq V(G)$ and $X_i\subseteq V(G)$ as in the statement of \cref{grow_unicycle}. 
  For each $\{i,j\}\in \binom{[p]}{2}$, 
  we apply \cref{double_unicycle} to $C_i$, $U_i$, $Y_i$, and $C_j$, $U_j$, $Y_j$. 
  If any of these applications of \cref{double_unicycle} produces a $d$-packing of $k$ cycles in $G$, then $G$ satisfies (a) and there is nothing to prove.
  Thus, for each $\{i,j\}\in \binom{[p]}{2}$ we obtain 
  $Y_{\{i,j\}}\subseteq V(G)$ and $X_{\{i,j\}}\subseteq V(G)$ as in the statement of \cref{double_unicycle}.
  Now, the unicycles $U_1,\ldots,U_p$ and the sets
  \begin{align*}
    Y & :=\textstyle\bigcup_{i\in[p]} Y_i\cup \bigcup_{\{i,j\}\in\binom{[p]}{2}} Y_{\{i,j\}}\text{, and} \\
    X & :=\textstyle\bigcup_{i\in[p]} X_i\cup \bigcup_{\{i,j\}\in\binom{[p]}{2}} X_{\{i,j\}}
  \end{align*}
  satisfy the conditions of \cref{all_the_ys}.
  In particular, $|X| < k\cdot(2k+s(k))+\binom{k}{2}\cdot s(k)$, as desired.
\end{proof}

We are now ready to prove (the following quantitative version of) \cref{thm:main-in-intro}.

\begin{thm}\label{thm:the-big-ball-of-wax}
  Let $f$ and $g$ be the following functions:
  \[
    \textstyle f(x)
      = 2x+2x^2 + (\binom{x}{2}+x)\cdot s(x) + \ell^\star(x+\tfrac{1}{2}s(x),3), \qquad
    g(x)= 19x.
  \]
  For all integers $k\ge 1$ and $d\ge 1$, for every graph $G$, either $G$ contains a $d$-packing of $k$ cycles or there exists $X\subseteq V(G)$ with $|X|\leq f(k)$ such that $G-B_G(X,g(d))$ is a forest.
\end{thm}

\begin{proof}
  Let $k\ge 1$ and $d\ge 1$ be integers and let $G$ be a graph. If $G$ contains a $d$-packing of $k$ cycles then there is nothing to prove. Thus, assume the opposite. Let $\set{\mathdefin{C_1,\ldots,C_p}}$ be a maximal $2d$-packing of cycles that are $d$-unicyclic in $G$.  Since $G$ has no $d$-packing of $k$ cycles, $p<k$.

  Consider the following iterative process, that constructs a sequence of cycles $\mathcal{D}':=\set{D'_1,D'_2,\ldots}$ such that $\mathcal{D}'$ is a $d$-packing of cycles each having length at most $6d+2$. Let $q$ be a nonnegative integer and suppose we have already constructed $D'_1,\ldots,D'_{q}$.

  If every cycle in $G$ contains a vertex in  $B_G(\bigcup_{i\in[p]} V(C_i),4d)\cup B_G(\bigcup_{i\in[q]}V(D'_i),4d)$ then we let the process end.  Otherwise, let $D$ be a cycle with no vertex in $B_G(\bigcup_{i\in[p]} V(C_i),4d)\cup B_G(\bigcup_{i\in[q]}V(D'_i),4d)$.   Apply \cref{short_or_unicycle_nearby} to $D$ to find a cycle $D'$ such that either
  \begin{enumerate*}[label=(\alph*),ref=\alph*]
    \item\label[p]{item:case-unicycle} $V(D')\subseteq B_G(V(D),2d)$ and $D'$ is $d$-unicyclic in $G$; or
    \item\label[p]{item:case-short-cycle} $V(D')\subseteq B_G(V(D),3d)$ and $D'$ is of length at most $6d+2$.
  \end{enumerate*}
  In case \cref{item:case-unicycle}, $V(D)\cap B_G(\bigcup_{i\in[p]} V(C_i),4d)=\emptyset$ and $V(D')\subseteq B_G(V(D),2d)$, which implies that $V(D')\cap B_G(\bigcup_{i\in[p]} V(C_i),2d)=\emptyset$. Therefore, $\set{C_1,\ldots,C_p,D'}$ is a $2d$-packing of $d$-unicyclic cycles in $G$, which contradicts the choice of $\set{C_1,\ldots,C_p}$.

  Therefore case \cref{item:case-short-cycle} applies.  By the choice of $D$, $\dist_G(V(D),V(D_i'))>4d$ for each $i\in[q]$.  Since $V(D')\subseteq B_G(V(D),3d)$, this implies that $\dist_G(V(D'),V(D'_i))>d$ for each $i\in[q]$. Set $D'_{q+1}:=D'$. Thus $\{D'_1,\ldots,D'_{q+1}\}$ is a $d$-packing of cycles in $G$. This completes the description of the process. Since $G$ has no $d$-packing of $k$ cycles, this process produces a family $\mathcal{D'}$ of size $q<k$. (Possibly $\mathcal{D'}=\emptyset$ and $q=0$.)

  Let
  \[
      \textstyle\mathdefin{Y_0}:=\bigcup_{i\in[q]}B_G(V(D'_i),4d).
  \]
  Let $\mathdefin{X_0}:=\{x_1,\ldots,x_{q}\}$ where  $x_i$ is an arbitrary vertex of $D'_i$ for each $i\in[q]$. Since $D_i'$ has length at most $6d+2$, we have $V(D_i')\subseteq B_G(x_i,3d+1)$ for each $i\in[q]$. Therefore,
  \begin{align}
    Y_0&\subseteq B_G(X_0,7d+1)\text{, and}\label{eq:M-contained-in-a-ball}\\
    |X_0|&< k.\label{eq:XM-size}
  \end{align}

  The stopping condition of the process described above ensures that
  every cycle in $G-Y_0$ contains a vertex in $\bigcup_{i=1}^p B_G(V(C_i),4d)$.

  Let
  \[
    \mathdefin{r}:=6d.
  \]

  Apply \cref{all_the_ys} to $\mathcal{C}:=\{C_1,\ldots,C_p\}$ to obtain unicyclic graphs $U_1,\ldots,U_p$, a set $Y_1\subseteq V(G)$ and a set $X_1\subseteq V(G)$ such that
  \begin{tightenum}
    \item for each $i\in[p]$, $U_i$ is a  $C_i$-rooted spanning BFS-unicycle of $G[B_G(V(C_i),r)]$;
    \item \label{item:Y_1_makes_unicyclic}
    for each $i\in[p]$, $Y_1$ contains at least one endpoint of each edge in $G[B_G(V(C_i),r)]- E(U_i)$;
  \end{tightenum}
  \begin{equation} 
  \label{eq:ball_intersectionsin_Y_1}
    B_{G}(V(C_i),r)\cap B_{G}(V(C_j),r)\subseteq Y_1
  \end{equation}
  for each $\set{i,j}\in\binom{[p]}{2}$;
  \begin{align}
    Y_1 & \subseteq B_G(X_1,2r+d)\text{, and} \label{eq:Yi-contained-in-a-ball}\\
    |X_1| & < \textstyle  2k^2+ (\binom{k}{2}+k)\cdot s(k). \label{eq:Xi-size}
  \end{align}
  For each $i\in[p]$, let \mathdefin{$y_i$} be an arbitrary vertex of $C_i$ and define $\mathdefin{Y_2}:=\mathdefin{X_2}:=\{y_1,\ldots,y_p\}$. Define
  \begin{align*}
    \mathdefin{\widehat{Y}}&:= \textstyle Y_0 \cup Y_1 \cup Y_2,\ \textrm{and}\\
    \mathdefin{\widehat{X}}&:=\textstyle X_0\cup X_1\cup X_2.
  \end{align*}
  Note that by~\eqref{eq:M-contained-in-a-ball} and \eqref{eq:Yi-contained-in-a-ball}, we have
  \begin{equation}\label{m_in_x_ball}
    \begin{split}
      \widehat{Y} & \textstyle\subseteq B_G(X_0,7d+1) \cup B_G(X_1,2r+d) \cup X_2\\
      & \subseteq B_G(\widehat{X}, 13d),
    \end{split}
  \end{equation}
  and by~\eqref{eq:XM-size} and \eqref{eq:Xi-size},  we have
  \begin{equation}
    \begin{split}
    |\widehat{X}|& \textstyle \leq |X_0| + |X_1| + |X_2|  \\
    &\textstyle< k + 2k^2+(\binom{k}{2}+k)\cdot s(k) + k.
    \end{split} \label{x_prime_size}
  \end{equation}
  Let
  \begin{align*}
    \mathdefin{F_0}
      & := \textstyle{G-\left(\widehat{Y}\cup \bigcup_{i\in[p]} B_G(V(C_i),r-2d)\right)}\text{, and}\\
    \mathdefin{F^-_0}
      & := \textstyle{G-\left(\widehat{Y}\cup \bigcup_{i\in[p]} B_G(V(C_i),r-d)\right)}. 
  \end{align*}
  Since every cycle in $G-Y_0$ contains a vertex in $\bigcup_{i\in[p]} B_G(V(C_i),r-2d)$ and $Y_0\subseteq \widehat{Y}$, $F_0$ is an induced forest in $G$. By definition, $F^-_0$  is an induced forest in $F_0$.

  For $i\in[p]$, an \defin{$i$-exit edge} is an edge of $G$ with one endpoint in $B_G(V(C_i),r-d)$ and one endpoint in $F_0^-$. A $4$-tuple $(e,i,e',j)$ is a \defin{good tuple} if
  \begin{itemize}[noitemsep,nosep]
    \item $e$ is an $i$-exit edge;
    \item $e'$ is a $j$-exit edge;
    and \item $e\neq e'$ and $e$ and $e'$ are incident to the same component of $F_0^-$.
  \end{itemize}
  (Note that possibly $i=j$.) Each good tuple $t:=(e,i,e',j)$  defines a walk $\mathdefin{W_t}:=P_1eP_0e'P_2$ in $G$ where
  \begin{itemize}[noitemsep,nosep]
    \item $P_1$ is the shortest path in $U_i$ from $V(C_i)$ to the endpoint of $e$ in $B_G(V(C_i),r-d)$;
    \item $P_0$ is the unique path in $F_0^-$ from the endpoint of $e$ in $F^-_0$ to the endpoint of $e'$ in $F^-_0$; and
    \item $P_2$ is the shortest path in $U_j$ from the endpoint of $e'$ in $B_G(V(C_j),r-d)$ to $V(C_j)$.
  \end{itemize}
  \begin{figure}[t]
  \begin{center}
  \includegraphics{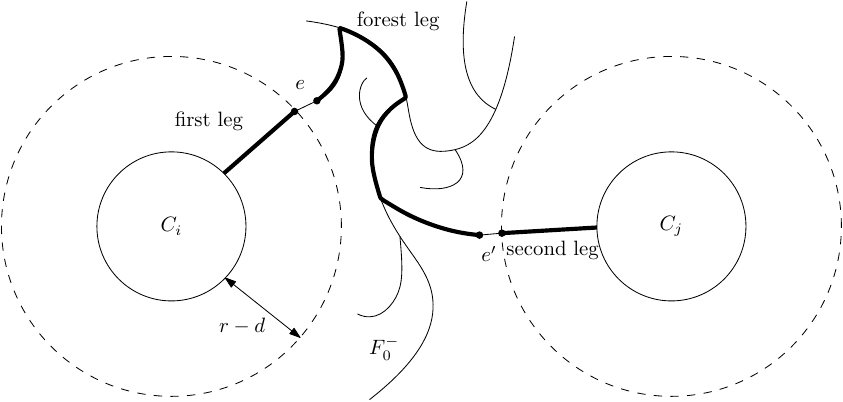}
  \end{center}
  \caption{A good tuple $t=(e,i,e',j)$ and the walk $W_t$.}
  \label{fig:good-tuple}
\end{figure}

  We call the path $P_0$ the \defin{forest leg} of $W_t$, $P_1$ the \defin{first leg} of $W_t$, and $P_2$ the \defin{second leg} of $W_t$.  The walk $eP_0e'$ is called the \defin{extended forest leg} of $W_t$.  Note that the first and second legs of $W_t$ each have exactly $r-d$ edges, while the forest leg of $W_t$ could be arbitrarily long.

  A tuple $t$ is \defin{admissible} if $t$ is good and $B_G(V(W_t),d) \cap \widehat{Y} = \emptyset$.  The following claim allows us to finish the proof by finding a set of balls that intersect the extended forest leg of each admissible tuple.

  \begin{clm}\label{hit_cycle}
    Let $C$ be a cycle in $G$.  Then  $C$ contains a vertex in $B_G(\widehat{Y},r)$ or $C$ contains the extended forest leg of some admissible tuple.
  \end{clm}

  \begin{clmproof}
    We may assume that $C$ contains no vertex in $B_G(\widehat{Y},r)$ since, otherwise, there is nothing to prove. Since $F_0^-$ is an induced forest in $G$, $C$ must contain some vertex not in $F_0^-$.  Therefore $C$ contains a vertex in $B_G(V(C_i),r-d)$  for some $i\in\{1,\ldots,p\}$. 
    Recall that $Y_1$ was obtained from applying \cref{all_the_ys} and that $G[B_G(V(C_i),r-d)]-Y_1$ is unicyclic by \eqref{item:Y_1_makes_unicyclic}. Thus, if $G[B_G(V(C_i),r-d)]-Y_1$ contains any cycle then this cycle is $C_i$.  Therefore  $G[B_G(V(C_i),r-d)]-(Y_1\cup \{y_i\})$ is a forest.  Since $C$ contains no vertex in $Y_1\cup\{y_i\}\subseteq \widehat{Y}$, $C$ must also contain a vertex not in $B_G(V(C_i),r-d)$.

    Therefore, $C$ contains an edge $v_0v_1$ with $v_0\in B_G(V(C_i),r-d)$ and $v_1\not\in B_G(V(C_i),r-d)$. 
    Since $d\geq1$, $v_1\in B_G(V(C_i),r)$. 
    Now, if $v_1\in B_G(V(C_j),r)$ for some $j\in[p]\setminus\set{i}$, then 
    $v_1 \in B_G(V(C_j),r) \cap B_G(V(C_i),r)\subseteq Y_1\subseteq \widehat{Y}$ by \eqref{eq:ball_intersectionsin_Y_1}, a contradiction.
    Therefore, $v_1\in V(F_0^-)$, 
    so $v_0v_1$ is an $i$-exit edge.

    Consider the maximal path $v_1,\ldots,v_{h-1}$ in $C$ that contains $v_1$ and that is contained in $F_0^-$. Let $v_h$ be the neighbour of $v_{h-1}$ in $C-\set{v_{h-2}}$.  (Possibly $v_h=v_0$.)  Since $v_h\notin \widehat{Y}$ and $v_h$ is not a vertex of $F_0^-$, $v_h\in B_G(V(C_j),r-d)$ for some $j\in\{1,\ldots,p\}$.   Then $v_{h-1}v_h$ is a $j$-exit edge. 
    Since $F_0^-$ is a forest, $v_1,\ldots,v_{h-1}$ is the unique path between $v_1$ and $v_{h-1}$ in $F_0^-$, and it follows that 
    $v_0,\ldots,v_{h}$ is the extended forest leg of $W_t$ for the good tuple $t:=(v_0v_1,i,v_{h-1}v_h,j)$.   If $t$ is admissible, then there is nothing more to prove.  If $t$ is not admissible, then $B_G(V(W_t),d)$ contains a vertex of $\widehat{Y}$, say $m$. Let $P_0=v_1,\ldots,v_{h-1}$ be the forest leg of $W_t$, let $P_1$ be the first leg of $W_t$ and let $P_2$ be the second leg of $W_t$. If $m\in B_G(V(P_1),d)$ then there is $u\in V(P_1)$ such that $\dist_G(m,u)\leq d$ and so $\dist_G(m,v_0)\le \dist_G(m,u) + \dist_G(u,v_0)\le d + (r-d)$. If $m\in B_G(V(P_2),d)$ then $\dist_G(m,v_h)\le r$.  Similarly, if $m\in B_G(V(P_0),d)$ then $\dist_G(m,V(P_0))\le d$.  In each case $B_G(m,r)$ contains a vertex in $V(C)\supseteq \set{v_0,\ldots,v_h}$, a contradiction. 
  \end{clmproof}

  Let $t=(e,i,e',j)$ be an admissible tuple with $P_1eP_0e'P_2:=W_t$. 
  Recall that $p=|\mathcal{C}|$. 
  Define the $(p+1)$-tuple $\mathdefin{\Psi(t)}:=(\Psi_0(t),\Psi_1(t),\ldots,\Psi_p(t))$ where
  \begin{align*}
    \mathdefin{\Psi_\ell(t)} &:=
    \begin{cases}
      B_{G}(V(P_0),d) & \textrm{if $\ell=0$,}\\
      B_{G}(V(P_1),d) & \textrm{if $\ell=i$ and $\ell\neq j$,}\\
      B_{G}(V(P_2),d) & \textrm{if $\ell\neq i$ and $\ell= j$,}\\
      B_{G}(V(P_1),d)\cup B_{G}(V(P_2),d) & \textrm{if $\ell= i$ and $\ell= j$,}\\
      \emptyset & \textrm{if $\ell\notin\{0,i,j\}$.}
    \end{cases}
  \end{align*}

  For each $i\in[p]$, let
  \begin{align*}
    \mathdefin{F_i} & := G[B_G(V(C_i),r)]-\widehat{Y},\\
    \mathdefin{F_i^-} & := G[B_G(V(C_i),r-d)]-\widehat{Y}.
  \end{align*}
  Since $\widehat{Y}$ contains $Y_1$, $F_i$ is a subgraph of $U_i$.  Since $\widehat{Y}$ contains $y_i\in V(C_i)$, $F_i$ is a forest.  Since $F^-_i$ is a subgraph of $F_i$, it is also a forest.

  \begin{clm}\label{clm:three-components}
    Let $t$ be an admissible tuple. Then
    \begin{tightenum}
      \item\label[p]{item:Psi0-connected} $G[\Psi_0(t)]$ is a connected subgraph of $F_0$,
      \item\label[p]{item:Psil-connected} $G[\Psi_\ell(t)]$ is a subgraph of $F_\ell$ and has $c_{\ell}$ components where $c_{\ell}\in\set{0,1,2}$, for each $\ell\in[p]$.  Moreover, $\sum_{\ell\in[p]}c_{\ell}\leq 2$.
    \end{tightenum}
  \end{clm}

  \begin{clmproof}\
    Let $t:=(e,i,e',j)$. Let $P_0$, $P_1$, and $P_2$ be the forest leg, the first leg, and second leg of $W_t$, respectively.   For the proof of~\cref{item:Psi0-connected}, recall that $P_0$ is a connected subgraph of $F_0^-$. Since $t$ is admissible, $B_G(V(P_0),d)$ contains no vertex in $\widehat{Y}$. Since $P_0\subseteq F_0^-$, $V(P_0)$ contains no vertex of $B_G(V(C_i),r-d)$ for each $i\in[p]$. Hence, $B_G(V(P_0),d)$ contains no vertex in $B_G(V(C_i),r-2d)$ for any $i\in[p]$. Therefore $B_G(V(P_0),d)=B_{F_0}(V(P_0),d)$. Since $P_0$ is connected in $F_0$, $B_{F_0}(V(P_0),d)$ induces a connected subgraph of $F_0$.

    Now we prove~\cref{item:Psil-connected}. If $\ell\in [p]\setminus\{i,j\}$, then $\Psi_\ell(t)$ is empty, so $c_\ell=0$. 
    Assume thus that $\ell\in \{i,j\}$. 
    We prove that $B_G(V(P_1), d)$ induces a connected subgraph of $F_i$ and that $B_G(V(P_2),d)$ induces a connected subgraph of $F_j$.  
    By the definition of $\Psi_\ell(t)$, this implies that $c_\ell=1$ in case $i\neq j$, and $c_\ell \leq 2$ in case $i=j$. 
    We prove the first of these statements, the proof of the second is symmetric.

    Since $P_1$ contains no vertex in $Y_1\cup\{y_i\}$, $P_1$ is a connected subgraph of $F_i^-$.  Since $V(F^-_i)\subseteq B_G(V(C_i),r-d)$, $B_G(V(P_1),d)\subseteq B_G(V(C_i),r)$.  Furthermore, $B_G(V(P_1),d)$ contains no vertex in $\widehat{Y}$ as $t$ is admissible, so $B_G(V(P_1),d)=B_{F_i}(V(P_1),d)$. Since $P_1$ is connected in $F_i$, $B_{F_i}(V(P_1),d)$ induces a connected subgraph of $F_i$.
  \end{clmproof}

  \begin{clm}\label{w_distance}
    Let $t_1$ and $t_2$ be admissible tuples. If $\Psi_\ell(t_1)\cap \Psi_\ell(t_2)=\emptyset$ for all $\ell\in\{0,\ldots,p\}$, then $\dist_G(V(W_{t_1}),V(W_{t_2}))> d$.
  \end{clm}

  \begin{clmproof}
    Assume for the sake of contradiction that $\Psi_\ell(t_1)\cap\Psi_\ell(t_2)=\emptyset$ for all $\ell\in\{0,\ldots,p\}$ but there exist $v\in V(W_{t_1})$ and $w\in V(W_{t_2})$ with $\dist_G(v,w)=\dist_G(V(W_{t_1}),V(W_{t_2})) \le d$.  Let $t_2=(e,\ell,e',\ell')$ and let $Q_0$, $Q_1$, and $Q_2$ be the forest, first, and second legs of $W_{t_2}$, respectively.

    We first consider the case in which $v$ is in the forest leg of $W_{t_1}$ or $w$ is in the forest leg of $W_{t_2}$.   Without loss of generality, suppose $v$ is in the forest leg $P_0$ of $W_{t_1}$. Then $w\in B_G(v,d)\subseteq B_G(V(P_0),d)\subseteq\Psi_0(t_1)$.  Now we will show that $w\in \Psi_0(t_2)$.   Recall that $w$ lies in one of the three legs of $W_{t_2}$. If $w\in V(Q_0)$ then $w\in V(Q_0)\subseteq \Psi_0(t_2)$ as desired. Now assume $w$ lies in the first or second leg of $W_{t_2}$. Suppose without loss of the generality that $w\in V(Q_1)$. Let $x$ be the endpoint of the $\ell$-exit edge $e$ that lies in $Q_0$. 
    Observe that the ball $B_G(w,d)$ hits $F^-_0$, since it contains $v$. 
    By the definition of $Q_1$, this implies that $x\in B_G(w,d)$. 
    But then $w\in B_G(x,d)\subseteq\Psi_0(t_2)$, and therefore $\Psi_0(t_1)\cap\Psi_0(t_2)\supseteq\{w\}\supsetneq\emptyset$, a contradiction.

    Next we consider the case where $v$ is not in the forest leg of $W_{t_1}$ and $w$ is not in the forest leg of $W_{t_2}$.  Without loss of generality, suppose $v$ is in the first leg $P_1$ of $W_{t_1}$ and $w$ is in the first leg $Q_1$ of $W_{t_2}$. Recall that $t_2=(e,\ell,e',\ell')$, so $Q_1\subseteq F_\ell^-$. Since $\dist_G(v,w) \le d$ by assumption, 
    we then have $v\in B_G(w,d)\subseteq B_G(V(Q_1),d) \subseteq \Psi_\ell(t_2)$. Since $t_2$ is admissible and $w\in V(F^-_\ell)$ we have that $B_G(w,d) \subseteq V(F_\ell)$ and so $v\in V(F_\ell)$. Since $v\not\in V(F_0^-)$ and $v\not\in \widehat{Y}$, $v\in B_G(V(C_j),r-d)$ for some $j\in\{1,\ldots,p\}$. Since $v\notin B_G(V(C_\ell),r)\cap B_G(V(C_j),r)$ for any $j\neq \ell$, it must be that $j=\ell$. Thus $v\in V(F^-_\ell)$.  Hence $v\in\Psi_\ell(t_1)$. Therefore, $\Psi_\ell(t_1)\cap\Psi_\ell(t_2)\supseteq\{v\}\supsetneq \emptyset$, a contradiction.
  \end{clmproof}

  \begin{clm}\label{hungarians_hit}
    Let $X\subseteq V(G)$ and let $t:=(e,i,e',j)$ be an admissible tuple with extended forest leg $eP_0e'$ and such that $X\cap \Psi_\ell(t)\neq\emptyset$ for some $\ell\in\{0,\ldots,p\}$.  Then $B_G(X,r)\cap V(eP_0e')\neq\emptyset$.
  \end{clm}

  \begin{clmproof}
    Note that $\Psi_{\ell}(t)$ is non-empty only for the values of $\ell\in\set{0,i,j}$. If $\ell=0$ then, by assumption, $X$ contains a vertex in $B_G(V(P_0),d)=\Psi_0(t)$.  Therefore $B_G(X,d)\subseteq B_G(X,r)$ contains a vertex of $P_0$ and there is nothing more to prove. Otherwise, let $P_1$ be the first leg of $W_t$ and let $P_2$ be the second leg of $W_t$. If $\ell=i$ then $\dist_G(X,V(P_1))\le d$.  Therefore $B_G(X,d)$ contains a vertex of $P_1$.  Therefore $B_G(X,r)$ contains all vertices of $P_1$, including the endpoint of $e$ in $P_1$.  If $\ell=j$, then the same argument shows that $B_G(X,r)$ contains the endpoint of $e'$ in $P_2$.
  \end{clmproof}

  Let $\mathdefin{F^{\star}}$ be the forest obtained by taking disjoint copies $\mathdefin{F_0^\star,\ldots, F_p^\star}$ of $F_0,\ldots, F_p$, respectively.  That is,  $F_0^\star,\ldots, F_p^\star$ are pairwise vertex-disjoint and $F_i^\star$ is isomorphic to $F_i$, for each $i\in\{0,\ldots,p\}$. For each $i\in\{0,\ldots,p\}$ and each $v\in V(F_i)$, define $\mathdefin{v_i}$ to be the unique vertex of $F^\star_i$ that corresponds to $v$.

  Let $t$ be an admissible tuple. Let $\mathdefin{\Psi^\star(t)}:=\bigcup_{i=0}^p \{v_i:v\in \Psi_i(t)\}$. By \cref{clm:three-components}, $F^\star[\Psi^\star(t)]$ has at most three components. Let
  \[
    \mathdefin{\mathcal{A}}:=\{F^\star[\Psi^\star(t)]:\text{$t$ is an admissible tuple}\}.
  \]

  By~\cref{thm:gyarfas-lehel-general} applied for $\mathdefin{k^\star}:=k+\tfrac{1}{2}s(k)$, $c=3$, $F^\star$, and $\mathcal{A}$, at least one of the following is true:
  \begin{enumerate*}[label=(\alph*),ref=\alph*]
    \item\label[p]{gyarfas-packs} there are $k^\star$ pairwise vertex-disjoint members of $\mathcal{A}$; or
    \item\label[p]{gyarfas-hits} there exists $X^\star\subseteq V(F^\star)$ with $|X^\star|\leq \ell^\star(k^\star,3)$ and such that $X^\star \cap V(A)\neq\emptyset$ for each $A\in \mathcal{A}$.
  \end{enumerate*}

  First, we consider outcome \cref{gyarfas-hits} and let $X^\star$ be the promised subset of $V(F^\star)$. Define
  \[
    \mathdefin{X_3}:=\set{v\in V(G):v_i\in X^\star \textrm{ for some $i\in\set{0,\ldots,p}$}}.
  \]
  Note that for every admissible tuple $t$, $X^\star$ intersects $\Psi^\star(t)$, and therefore
  \begin{equation}\label{eq:X_H}
    X_3 \cap \Psi_i(t) \neq\emptyset,\ \textrm{for some $i\in\set{0,\ldots,p}$.}
  \end{equation}
  Let
  \[
  \mathdefin{X}:=\widehat{X} \cup X_3. 
  \]
  By~\eqref{x_prime_size}
  \[
    |X| \leq |\widehat{X}| + |X_3| < \textstyle 2k + 2k^2 + (\binom{k}{2}+k)\cdot s(k) +  \ell^\star(k+\tfrac{1}{2}s(k),3) = f(k).
  \]
  We now show that every cycle $C$ in $G$ contains a vertex of
  \[
    B_G(\widehat{X},19d)\cup B_G(X_3,r)\subseteq B_G(X,19d) = B_G(X,g(d)).
  \]
  Consider an arbitrary cycle $C$ in $G$. By~\cref{hit_cycle}, either $C$ contains a vertex in $B_G(\widehat{Y},r)$ or $C$ contains the extended forest leg of some admissible tuple $t$. In the former case, by \eqref{m_in_x_ball}, $\widehat{Y}\subseteq B_G(\widehat{X},13d)$ so $B_G(\widehat{Y},r)\subseteq B_G(\widehat{X},19d)$ and we are done. Therefore, we can assume that $C$ contains the extended forest leg of $W_t$ for some admissible tuple $t$. By~\eqref{eq:X_H} and~\cref{hungarians_hit}, $B_G(X_3,r)$ contains a vertex of $C$.

  Now we consider outcome \cref{gyarfas-packs}: there are $k^\star$ pairwise vertex-disjoint members $\mathdefin{A_1,\ldots,A_{k^\star}}$ of $\mathcal{A}$. For each $j\in[k^\star]$, let $\mathdefin{t_j}$ be an admissible tuple such that $A_j = F^\star[\Psi^\star(t_j)]$. Let $\mathdefin{T}:=\set{t_1,\ldots, t_{k^\star}}$. Therefore, for every $i\in\{0,\ldots,p\}$ and every $t,t'\in T$ with $t\neq t'$ we have $\Psi_i(t) \cap \Psi_i(t') = \emptyset$. By \cref{w_distance},
  \begin{equation}\label{eq:Wt-and-Wt-far-apart}
    \dist_G(V(W_{t}), V(W_{t'}))> d,
  \end{equation}
  for all $t,t'\in T$ with $t\neq t'$.

  Recall that, for each $t\in T$, $W_t$ is a path in $G$ or contains a cycle in $G$. Let $\mathdefin{J}:=\{t\in T:\text{$W_t$ contains a cycle}\}$.  By~\eqref{eq:Wt-and-Wt-far-apart} the cycles defined by $W_t$ for $t\in J$ define a $d$-packing of cycles in $G$.  Therefore $|J|<k$ and $|T\setminus J|> k^\star-k= \tfrac{1}{2}s(k)$.

  Let $\mathcal{C}'$ be the subset of $\{C_1,\ldots,C_p\}$ that contains each $C_i$ if and only if $V(C_i)$ contains an endpoint of $W_t$ for some $t\in T\setminus J$.    
  Define the graph $G':=\bigcup_{C\in\mathcal{C}'} C\cup \bigcup_{t\in T\setminus J} W_{t}$. 
  Recall that the cycles of $\{C_1,\ldots,C_p\}$ are vertex disjoint by choice, and that all $W_t$'s with $t\in T$ are vertex disjoint by \ref{eq:Wt-and-Wt-far-apart}. 
  Thus all vertices of $G'$ are of degree $2$ or $3$.  Since $G'$ does not contain $W_t$ for any $t\in J$, each degree $3$ vertex of $G'$ is an endpoint of $W_{t}$ for some $t\in T\setminus J$. Therefore, $G'$ contains $2|T\setminus J|\geq s(k)$ vertices of degree $3$. By \cref{thm:simonovits}, $G'$ contains a set $\mathcal{D}$ of $k$ pairwise vertex-disjoint cycles.  The degree-$3$ vertices of $G'$ partition the edges of each cycle in $\mathcal{D}$ into paths, where each such path is either contained in a cycle $C_i$ for some $i\in[p]$ or is a path $W_t$ for some $t\in T\setminus J$.

  We now show that $\mathcal{D}$ is a $d$-packing of cycles in $G$. Let $D$ and $D'$ be two distinct cycles in $\mathcal{D}$.  Let $P$ be a shortest path in $G$ from $V(D)$ to $V(D')$.  Let $v\in V(D)$ and $v'\in V(D')$ be the endpoints of $P$.  If $v\in W_t$ and $v'\in W_{t'}$ for some $t,t'\in T$ such that $W_t\subseteq D$ and $W_{t'}\subseteq D'$ then, since $D$ and $D'$ are vertex-disjoint we have $t\neq t'$, and by \eqref{eq:Wt-and-Wt-far-apart} we have $\dist_G(v,v')\ge \dist_G(V(W_t),V(W_{t'}))>d$, as desired.

  Therefore, without loss of generality we may assume that $v\in V(C_i)$ for some $i\in[p]$ and that $v\notin V(W_t)$ for any $t\in T$ with $W_t\subseteq D$.  Since $P$ is a shortest path between $V(D)$ and $V(D')$, we have $E(P)\cap E(D)=\emptyset$ and $E(P)\cap E(D')=\emptyset$.

  If the first edge of $P$ is contained in $C_i$ then the first edge of $P$ and the two edges of $D$ incident to $v$ are contained in $G'$.  Therefore $v$ is the endpoint of $W_t$ for some admissible tuple $t\in T\setminus J$ and such that $W_t\subseteq D$, which contradicts the choice of $v$.

  Suppose now that the first edge of $P$ is not contained in $C_i$.  We may assume that $V(P)\subseteq B_G(V(C_i),d)$ since otherwise $\len(P)>d$.  Since $V(P)\subseteq B_G(V(C_i),d)$ and $C_i$ is $d$-unicyclic, $v'\not\in V(C_i)$.  Since $\set{C_1,\ldots,C_p}$ is a $2d$-packing of cycles in $G$, $v'\not\in V(C_j)$ for any $j\in[p]$. Therefore $v'$ is in $W_{t'}$ for some admissible tuple $t'$ and $W_{t'}\subseteq D'$.

  Since $v'\in B_G(V(C_i),d)\subseteq B_G(V(C_i),r-d)$, $v'$ is not in the forest leg of $W_{t'}$. Assume, without loss of generality, that $v'$ is in the first leg $P_1'$ of $W_{t'}$. Since $t'$ is admissible and by~\eqref{eq:ball_intersectionsin_Y_1}, 
  $P_1'$ contains no vertex of $\bigcup_{j\in[p]\setminus \set{i}} (B_G(V(C_i),r)\cap B_G(V(C_j),r)$. Therefore, $P_1'$ connects a vertex of $V(C_i)$ with an $i$-exit edge.  Since $C_i$ is $d$-unicyclic and the first edge of $P$ is not an edge of $C_i$, $P_1'$ contains the first edge of $P$. Therefore, $P\subseteq P_1'\subseteq W_{t'}\subseteq D'$.  Since $v\in V(D)$ is incident to the first edge of $P$, this contradicts the fact that $D$ and $D'$ are vertex disjoint. This completes the proof that $\mathcal{D}$ is a $d$-packing in $G$. Since $|\mathcal{D}|\geq k$ this is the final contradiction that completes the proof of \cref{thm:the-big-ball-of-wax}.
\end{proof}

\section{Proof of \cref{thm:gyarfas-lehel-general}}
\label{sec:hungarians}

In this section, we present a complete proof of  \cref{thm:gyarfas-lehel-general} and show that, for any fixed $c$, the bounding function $\ell^\star(k,c)$ is polynomial in $k$.  All of the ideas from this proof come from \citet{gyarfas.lehel:helly}, who proved it when the host forest $F$ is a path and sketched changes needed for the general case.

Let $F$ be a forest. For each component of $F$, fix an arbitrary vertex of $F$. This makes $F$ a \defin{rooted forest}. For each component $C$ in an arbitrary subgraph $H$ of $F$, there is a unique vertex $v_C$ that is closest to a root of $F$. We call $v_C$ the \defin{root} of $C$ in $F$.  We treat $H$ as a rooted forest, where each component $C$ of $H$ is rooted at $v_C$. For two vertices $u$ and $v$ in $F$, if $v$ is contained in the unique path from $u$ to a root of $F$, then we say that $u$ is a \defin{descendant} of $v$ in $F$, and when additionally $u\neq v$  then $u$ is a \defin{strict descendant} of $v$ in $F$.  If $u$ is a strict descendant of $v$ and $uv\in E(F)$ then $u$ is a \defin{child} of $v$. For a vertex $v$ in $F$ the \defin{subtree}, $\mathdefin{F(v)}$, of $v$ is the subgraph of $F$ induced by all the descendants of $v$ in $F$ that is rooted at $v$. 
Let $X\subseteq V(F)$. 
We say that $x\in X$ is \defin{minimal} in $X$ with respect to $F$ if no strict descendant of $x$ lies in $X$. 


For an indexed family $A:=(A_i)_{i\in I}$ with index set $I$, we use $\mathdefin{\pi_i(A)}$ to denote the element $A_i$ of $A$ indexed by $i$.
Let $I$ be a set
and let $F_i$ be a forest for each $i\in I$.\footnote{In our application, the forests in $(F_i)_{i\in I}$ are pairwise vertex-disjoint, but this is not required in the following results.}
A sequence $(A_i)_{i\in I}$ is a \mathdefin{$(F_i)_{i\in I}$-tuple} if, for each $i\in I$, $A_i$ is a null graph or $A_i$ is a connected subgraph of $F_i$.  An $(F_i)_{i\in I}$-tuple $(A_i)_{i\in I}$ is \defin{trivial} if $A_i$ is a null graph for each $i\in I$ and is \defin{non-trivial} otherwise. Let $A=(A_i)_{i\in I}$ and $B=(B_i)_{i\in I}$ be two $(F_i)_{i\in I}$-tuples. We say that $A$ and $B$ are \defin{independent} if $V(A_i)\cap V(B_i)=\emptyset$ for each $i\in I$. A set $\mathcal{A}$ of $(F_i)_{i\in I}$-tuples is \defin{independent} if the elements of $\mathcal{A}$ are pairwise independent.

We make use of the following corollary, which is obtained immediately from \cref{cor:general-systems-lemma}, below, by  taking $m=k$ and $x_1=x_2=\cdots=x_k=1$.

\begin{cor}\label{cor:interface-for-systems-lemma}
  Let $c,k$ be positive integers. Let $F_i$ be a forest for each $i\in[c]$. For each $j\in [k]$, let $\mathcal{A}_j$ be a set of independent $(F_i)_{i\in[c]}$-tuples with $|\mathcal{A}_j|\geq k^c$. Then, for each $j\in[k]$ there exists $B_j \in \mathcal{A}_j$ such that $B_1,\ldots,B_k$ are pairwise independent.
\end{cor}

To establish \cref{cor:interface-for-systems-lemma} we prove the following more general result, which is amenable to a proof by induction.

\begin{lem}\label{cor:general-systems-lemma}
Let $c,m,k$ be positive integers. Let $F_i$ be a forest for each $i\in[c]$. For each $j\in [m]$, let $\mathcal{A}_j$ be a set of independent $(F_i)_{i\in[c]}$-tuples with $|\mathcal{A}_j|\geq m^{c-1}k$. Let $x_1,\ldots,x_m$ be nonnegative integers with $x_1+\cdots+x_m=k$. Then, for each $j\in[m]$ there exist $A_j^{1},\ldots,A_j^{x_j}\in \mathcal{A}_j$ such that
\begin{equation}
  A_1^{1},\ldots,A_1^{x_1},A_2^{1},\ldots,A_2^{x_2},\ldots,A_m^{1},\ldots,A_m^{x_m}  \label{general-systems-output}
\end{equation}
are pairwise independent.
\end{lem}

\begin{proof}
  We call a tuple $(c,m,k,(F_i)_{i\in[c]},(\mathcal{A}_j)_{j\in[m]},(x_j)_{j\in[m]})$ that satisfies the assumptions of the lemma a \defin{$(c,m,k)$-instance} and we say that the sequence of sets in \eqref{general-systems-output} \defin{satisfies} this $(c,m,k)$-instance.

  Without loss of generality, we may assume that $x_j>0$ for each $j\in[m]$.  To see why this is so, let $J:=\{j\in[m]:x_j>0\}$ and $m':=|J|$. Since $m'\le m$, $\mathcal{I}':=(c,m',k,(F_i)_{i\in[c]},(\mathcal{A}_j)_{j\in J},(x_j)_{j\in J})$ is a $(c,m',k)$-instance with $x_j>0$ for each $j\in J$, and any sequence $(A_{j}^\ell)_{j\in I,\, \ell\in [x_j]}$ that satisfies this instance also satisfies \eqref{general-systems-output}.  From this point on we assume that $x_j>0$ for each $j\in[m]$.

  Without loss of generality, we assume that for all $j\in[m]$, $i\in[c]$, $A\in\mathcal{A}_j$, the tree $\pi_i(A)$ is non-null. Indeed, otherwise for all $j$, $i$, $A$ violating this rule, we introduce a new vertex in $F_i$ and make $\pi_i(A)$ be the graph that contains only this new vertex.  Since this new vertex appears only in $\pi_i(A)$, the set $\mathcal{A}_j$ is still independent.  Because $\mathcal{A}_j$ is independent for each $j\in[m]$, this assumption ensures that $\pi_i(A)\neq\pi_i(A')$ for each $i\in[c]$, $j\in[m]$ and distinct $A,A'\in \mathcal{A}_j$. This ensures that, for every non-empty $I\subseteq[c]$, every $j\in[m]$, and every $\tilde{\mathcal{A}}_j\subseteq\mathcal{A}_j$, we have $|\{(\pi_i(A))_{i\in I}: A\in\tilde{\mathcal{A}}_j\}|=|\tilde{\mathcal{A}}_j|$. In the following we will use this fact repeatedly without explicitly referring to it.

  The assumption above, that for all $j\in[m]$, $i\in[c]$, $A\in\mathcal{A}_j$, the tree $\pi_i(A)$ is non-null, has another consequence, which we explain now. 
  Let $I$ be a non-empty subset of $[c]$, 
  for each $j\in[m]$ let $\mathcal{B}_j$ be obtained from a subset of $\mathcal{A}_j$ by restricting the tuples to the coordinates in $I$. 
  Let $k'$ be a positive integer and let $x'_1, \dots, x'_m$  be nonnegative integers with $x'_1 + \dots + x'_m=k'$. 
  Let $\mathcal{I}_{\mathcal{B}}:=(|I|,m,k',(F_i)_{i\in I},(\mathcal{B}_j)_{j\in[m]},(x'_j)_{j\in[m]})$ be a $(|I|,m,k')$-instance.  Then, for each $B\in\bigcup_{j\in[m]}\mathcal{B}_j$, the $(F_i)_{i\in I}$-tuple $(\pi_i(B))_{i\in I}$ is non-trivial.  Any two non-trivial independent $(F_i)_{i\in I}$-tuples are necessarily distinct.  It follows that any sequence that satisfies the instance $\mathcal{I}_{\mathcal{B}}$ is a sequence of $k'$ \emph{distinct} $(F_i)_{i\in I}$-tuples.  
  (Note that this would not necessarily be true if trivial tuples could appear.) 
  Treating this sequence as a set does not change its size.  Similarly, if there exist sets $\mathcal{B}_1',\ldots,\mathcal{B}_m'$ with $\mathcal{B}_j'\subseteq \mathcal{B}_j$ and $|\mathcal{B}_j'|=x_j'$ for each $j\in[m]$ and such that $\bigcup_{j\in[m]}\mathcal{B}_j'$ is independent and has size $k'$ then there exists a sequence that satisfies the instance $\mathcal{I}_{\mathcal{B}}$.  In this case, we say that $\mathcal{B}_1',\ldots,\mathcal{B}_m'$ \defin{satisfies} $\mathcal{I}_{\mathcal{B}}$.

  With these two assumptions justified, we now proceed with the proof.  Let $\mathcal{I}:=(c,m,k,(F_i)_{i\in[c]},(\mathcal{A}_j)_{j\in[m]},(x_j)_{j\in[m]})$ be a $(c,m,k)$-instance satisfying the assumptions in the previous paragraphs.  Our goal is to show the existence of a sequence of sets $\mathcal{A}'_1,\ldots,\mathcal{A}'_m$ that satisfies the instance $\mathcal{I}$.  We proceed by induction on $(c,k)$ in lexicographic order.

  Suppose first that $(c,k)=(1,1)$. Then $m=1$ and $|\mathcal{A}_1|\ge m^{c-1} k=1$.  Let $\mathcal{A}'_1$ be a $1$-element subset of $\mathcal{A}_1$.  Then the one element sequence $\mathcal{A}'_1$ satisfies $\mathcal{I}$.

  Now suppose that $c=1$ and $k>1$. For each $i\in[c]$, fix an arbitrary vertex in each component of $F_i$ to be a root. Let $\check{A}$ be an element of $\bigcup_{j\in[m]}\mathcal{A}_j$ such that the root, $v$, of $\pi_1(\check{A})$ has no strict descendant that is a root of $\pi_1(A)$ for any $A\in\bigcup_{j\in[m]}\mathcal{A}_j$. Let $\ell\in[m]$ be such that $\check{A}\in\mathcal{A}_{\ell}$.

  For each $j\in[m]$, let $\mathcal{B}_j:=\set{A\in\mathcal{A}_j\mid v\not\in V(\pi_1(A))}$. Since $\mathcal{A}_j$ is independent, there is at most one $A\in\mathcal{A}_j$ such that $\pi_1(A)$ contains $v$, and therefore,
  \begin{equation}
    |\mathcal{B}_j|\geq |\mathcal{A}_j|-1\geq k-1,\ \textrm{for each $j\in[m]$}. \label{eq:A-greater-k-1}
  \end{equation}
  Note that
  \begin{enumerate*}[label=(\arabic*)]
    \item $\pi_1(\check{A})\subseteq F_1(v)$,
    \item there is no $A\in\mathcal{B}_j$ such that the root of $\pi_1(A)$ is in $F_1(v)$, and
    \item for every $A\in\mathcal{B}_j$ we have $v\notin\pi_1(A)$.
  \end{enumerate*}
  We conclude that $\pi_1(B)$ and $\pi_1(\check{A})$ are vertex-disjoint for each $j\in[m]$ and each $B\in\mathcal{B}_j$.

  Define $y_j:=x_j$ for each $j\in[m]\setminus\set{\ell}$ and $y_\ell:=x_{\ell}-1$. Now, \eqref{eq:A-greater-k-1} implies that $\mathcal{I}_\mathcal{B}:=(1,m,k-1,(F_i)_{i\in[c]},(\mathcal{B}_j)_{j\in[m]},(y_j)_{j\in[m]})$ is a $(1,m,k-1)$-instance.  By the inductive hypothesis, there exists $\mathcal{B}_1',\ldots,\mathcal{B}_m'$ that satisfy $\mathcal{I}_\mathcal{B}$. For each $j\in[m]\setminus\{\ell\}$, let $\mathcal{A}'_j:=\mathcal{B}'_j$ and let $\mathcal{A}'_\ell:= \{\check{A}\} \cup \mathcal{B}'_\ell$. Then the sequence $\mathcal{A}'_1,\ldots,\mathcal{A}'_m$ satisfies $\mathcal{I}$. This completes the proof of the case $c=1$.

  Now suppose that $c>1$. Let $\mathcal{B}_j=\set{\pi_1(A)\mid A\in\mathcal{A}_j}$ and $y_j=m^{c-2}k$ for each $j\in[m]$. Observe that $|\mathcal{B}_j|=|\mathcal{A}_j|\geq m^{c-1}k$ for each $j\in[m]$. Therefore,
  \[
    \mathcal{I}_{\mathcal{B}}:=(1,m,m^{c-1}k,(F_1),(\mathcal{B}_j)_{j\in[m]},(y_j)_{j\in[m]})
  \]
  is a $(1,m,m^{c-1}k)$-instance.  Applying induction on $\mathcal{I}_\mathcal{B}$ we obtain a sequence of sets $\mathcal{B}_1',\ldots,\mathcal{B}_m'$ that satisfies the instance $\mathcal{I}_{\mathcal{B}}$.

  For each $j\in[m]$, let $\mathcal{C}_j=\set{(\pi_2(A),\ldots,\pi_c(A))\mid A\in\mathcal{A}_j, \pi_1(A)\in \mathcal{B}'_j}$. Observe that $|\mathcal{C}_j|=|\mathcal{B}'_j|=y_j=m^{c-2}k$.  Therefore, the tuple $\mathcal{I}_{\mathcal{C}}:=(c-1,m,k,(F_i)_{i\in\{2,\ldots,c\}},(\mathcal{C}_j)_{j\in[m]},(x_j)_{j\in[m]})$ is a $(c-1,m,k)$-instance.  Applying induction yields $\mathcal{C}_1',\ldots,\mathcal{C}_m'$ that satisfies the instance $\mathcal{I}_{\mathcal{C}}$.

  For each $j\in[m]$, define $\mathcal{A}'_j:=\bigcup_{C\in\mathcal{C}'_j}\{A\in\mathcal{A}_j\mid (\pi_2(A),\ldots,\pi_c(A))=C\}$.  Since $\mathcal{C}'_j\subseteq\mathcal{C}_j$ is independent,  $|\mathcal{A}'_j|=|\mathcal{C}'_j|=x_j$. Furthermore, $|\bigcup_{j\in[m]}\mathcal{A}'_j|=|\bigcup_{j\in[m]}\mathcal{C}'_j|=k$. All that remains is to show that $\bigcup_{j\in[m]}\mathcal{A}'_j$ is independent. For any two distinct $A,A'\in\bigcup_{j\in[m]}\mathcal{A}'_j$, the graphs $\pi_1(A)$ and $\pi_1(A')$ are vertex disjoint because $(\pi_1(A))$ and $(\pi_1(A'))$ each appear in $\bigcup_{j\in[m]}\mathcal{B}'_j$.  For each $i\in\{2,\ldots,c\}$, $\pi_i(A)$ and $\pi_i(A')$ are vertex disjoint because $(\pi_2(A),\ldots,\pi_c(A))$ and $(\pi_2(A'),\ldots,\pi_c(A'))$ each appear in $\bigcup_{j\in[m]}\mathcal{C}'_j$.  Thus, $\bigcup_{j\in[m]}\mathcal{A}'_j$ is independent, so $\mathcal{A}_1,\ldots,\mathcal{A}'_m$ satisfies $\mathcal{I}$.  This completes the proof of the lemma.
\end{proof}

\begin{lem}\label{lem:hungarians-Fi-tuples}
There exists a function $\ell:\mathbb{N}^2\to\mathbb{N}$ such that, for all $k,c\in\N$ with $k,c\ge 1$, for every $(F_i)_{i\in[c]}$ where $F_i$ is a forest for each $i\in[c]$, for every set $\mathcal{A}$ of non-trivial $(F_i)_{i\in[c]}$-tuples either
   \begin{tightenum}
     \item\label[p]{item:hungarian-support:independent-set} there are $k$ members of $\mathcal{A}$ that form an independent set; or
     \item\label[p]{item:hungarian-support:hitting-set} there exists a sequence of sets $(X_i)_{i\in[c]}$ with $X_i \subseteq V(F_i)$ for each $i\in[c]$, $\sum_{i\in[c]}|X_i|\leq \ell(k,c)$  and such that for each  $(A_1,\ldots,A_c)\in\mathcal{A}$,
     there exists $i\in[c]$ with
     $X_i\cap V(A_i)\neq\emptyset$.
   \end{tightenum}
\end{lem}

\begin{proof}
  We will show that the lemma holds with 
  \[
  \ell(k, c):= c\cdot \left(\prod_{j=1}^{c-1}2^{j!}\right)\, k^{c!}. 
  \]
  In order to help the induction go through, we prove the lemma with the following slightly more precise property in place of \cref{item:hungarian-support:hitting-set}:   
\begin{tightenum}[label={\normalfont (b')}]
     \item\label[p]{item:hungarian-support:hitting-set_precise} 
     there exists a sequence of sets $(X_i)_{i\in[c]}$ with $X_i \subseteq V(F_i)$ and      
   $|X_i|\le \left(\prod_{j=1}^{c-1}2^{j!}\right)\, k^{c!}$
     for each $i\in[c]$, 
     and such that for each  $(A_1,\ldots,A_c)\in\mathcal{A}$,
     there exists $i\in[c]$ with
     $X_i\cap V(A_i)\neq\emptyset$.
   \end{tightenum}
  Note that this implies that  $\sum_{i\in[c]}|X_i|\le c\cdot (\prod_{j=1}^{c-1}2^{j!})\, k^{c!} = \ell(k, c)$, so \cref{item:hungarian-support:hitting-set} follows from \ref{item:hungarian-support:hitting-set_precise}.  
  
  We begin with several assumptions, each without loss of generality, about $(F_i)_{i\in[c]}$ and $\mathcal{A}$.

  We may assume that $\pi_i(A)$ contains at least one vertex, for each $A\in\mathcal{A}$ and each $i\in[c]$. Indeed, otherwise we introduce a new vertex in $F_i$ and make $\pi_i(A)$ be the graph that contains only this new vertex. Since this vertex appears only in $\pi_i(A)$, this does not affect outcome \cref{item:hungarian-support:independent-set}. If this new vertex appears in the set $X_i$ promised by outcome \ref{item:hungarian-support:hitting-set_precise}, then we can remove it from $X_i$ and replace it with any other vertex in $\bigcup_{j\in[c]\setminus\{i\}}V(\pi_j(A))$.  (The fact that $A$ is non-trivial for each $A\in\mathcal{A}$ ensures that $\bigcup_{j\in[c]\setminus\{i\}}V(\pi_j(A))$ contains at least one vertex.)  This assumption ensures that for any non-empty $I\subseteq[c]$, and any $A\in\mathcal{A}$, $(\pi_i(A))_{i\in I}$ is non-trivial.

  We may assume that $F_i$ is a tree, for each $i\in[c]$.  We do this by introducing a new vertex $r_i$ adjacent to one vertex in each component of $F_i$.  We treat $F_i$ as a rooted tree rooted at $r_i$.

  Finally, we claim that we may assume that for each $i\in[c]$, $F_i$ is a binary tree,\footnote{A \defin{binary tree} is a rooted tree in which each vertex has at most two children. The root of a binary tree is not considered to be a leaf, even if it has no children.}, thanks to the following argument.
  For each $i\in[c]$, and each vertex $v$ in $F_i$ with children $u_1,\ldots,u_\beta$, define $T_i(v)$ to be a binary tree rooted at $v$ and whose leaves are $u_1,\ldots,u_\beta$, and let $T^-_i(v):=T_i(v)-\{u_1,\ldots,u_\beta\}$. 
  As expected, we choose these binary trees so that every two are vertex-disjoint except possibly for the one common vertex imposed by the definition; that is, given two distinct vertices $v,w$ in $F_i$, if $vw$ is not an edge in $F_i$ then $T_i(v)$ and $T_i(w)$ are vertex-disjoint, and if $v$ is the parent of $w$ then their only common vertex is $w$, which is then a leaf of $T_i(v)$.    
  Now let $F_i':=\bigcup_{v\in V(F_i)} T_i(v)$ and note that $F_i'$ is a binary tree for each $i\in[c]$.
  For each $A\in\mathcal{A}$, define $A':=(A'_1,\ldots,A'_c)$ where $A'_i:=F_i'[\bigcup_{v\in V(\pi_i(A))} V(T^-_i(v))]$ for each $i\in [c]$ and define $\mathcal{A}':=\{A'\mid A\in\mathcal{A}\}$.  Then $\mathcal{A}'$ is a set of non-trivial $(F_i')_{i\in[c]}$-tuples. Observe that for all $A,B\in\mathcal{A}$,
  \begin{enumerate*}[label=(\arabic*),ref=\arabic*]
    \item\label[p]{preserves-independence} $A$ and $B$ are independent if and only if $A'$ and $B'$ are independent;
    \item\label[p]{preserves-hitting} if $A_i'$ contains a vertex in $T_i^-(v)$ for some $i\in[c]$ and $v\in V(F_i)$, then $A_i$ contains $v$.
  \end{enumerate*} 
  Now, if we know that \cref{lem:hungarians-Fi-tuples} holds in the special case where the forests are binary trees, we may proceed as follows:  We apply the lemma to $(F'_i)_{i\in[c]}$ and $\mathcal{A}'$.  If this results in an independent set $I'\subseteq\mathcal{A}'$ of size $k$, then \cref{preserves-independence} ensures that $I:=\{A\in\mathcal{A}\mid A'\in I'\}$ satisfies \cref{item:hungarian-support:independent-set}.  If this results in sets $X_1',\ldots,X_c'$, then define $X_i:=\{v\in V(F_i)\mid V(T_i^-(v))\cap X_i'\neq\emptyset\}$ for each $i\in[c]$.  Then \cref{preserves-hitting} ensures that the sets $X_1,\ldots,X_c$ satisfy \ref{item:hungarian-support:hitting-set_precise}. 
  Therefore, to prove \cref{lem:hungarians-Fi-tuples} in full, it is enough to prove it in the case of binary trees, as claimed. 
  Assume thus from now on that $F_i$ is a binary tree for each $i\in[c]$. 

  We now proceed by induction on $c$.  In order to avoid presenting an argument for $c=1$ that is a special case of the argument for $c>1$, we treat $c=0$ as the base case of the induction.  In this case, the conclusion of the lemma is trivial since the empty sequence satisfies the conditions of outcome \ref{item:hungarian-support:hitting-set_precise}.

  Now suppose that $c\ge 1$. 
  Consider the following iterative process, that constructs a subset of vertices $\set{v_1,\ldots,v_m}$ of $F_1$. 
  Suppose that $v_1,\ldots,v_{z-1}$ 
  have already been defined, for some $z\ge 1$. 
  If $r_1\in\set{v_1,\ldots,v_{z-1}}$, then we set 
  $m:=z-1$ and stop the process. 
  Otherwise, define  
\[
F_{(z)}=F_1-\set{v_1,\ldots,v_{z-1}}.
\]
Recall that $F_{(z)}$ is a rooted forest and that, for each vertex $v$ of that forest, $F_{(z)}(v)$ denotes the subgraph of $F_{(z)}$ induced by all the descendants of $v$ in $F_{(z)}$ that is rooted at $v$. 
Now for each vertex $v$ of $F_{(z)}(r_1)$ let 
      \begin{align}
      \mathcal{B}_{z}(v)&=\set{A\in \mathcal{A}\mid \pi_1(A)\subseteq F_{(z)}(v)},\notag\\ 
      \mathcal{C}_{z}(v)&=\set{(\pi_2(A),\ldots,\pi_c(A))\mid A\in\mathcal{B}_{z}(v)}.\notag
  \end{align}
If $c=1$ and $\mathcal{B}_z(r_1)=\emptyset$, then we set 
$m:=z-1$ and stop the process.
If $c\geq2$ and $\mathcal{C}_z(r_1)$ contains no independent set of size $k^{c-1}$, then we set $m:=z-1$ and stop the process.
Otherwise, we define $v_z$ to be a minimal vertex in $F_{(z)}(r_1)$ that satisfies one of the following properties:
  \begin{enumerate}
      \item $c=1$ and $\mathcal{B}_{z}(v) \neq \emptyset$; 
      \item $c\geq 2$ and $\mathcal{C}_{z}(v)$ contains an independent set of size $k^{c-1}$.
  \end{enumerate}
This completes the definition of the process.

  We consider two separate cases:
  \begin{enumerate*}[label=(\arabic*)]
    \item $m\geq k$, in which case we show that $\mathcal{A}$ contains an independent set of size $k$, so \eqref{item:hungarian-support:independent-set} holds;
    \item $m<k$, in which case we show that there exist $X_1,\ldots,X_c$ that satisfy \eqref{item:hungarian-support:hitting-set_precise}. 
  \end{enumerate*}

  Suppose first that $m\geq k$ so that $v_1,\ldots,v_k$ is well-defined.  For each $z\in[k]$, define $\mathcal{B}_{z}:=\mathcal{B}_{z}(v_z)$. Consider $\alpha,\beta\in[k]$ with $\alpha<\beta$. 
  Note that $v_{\alpha}$ is not a vertex of $F_{(\beta)}$ and therefore 
  \begin{equation}
    V(F_{(\alpha)}(v_{\alpha})) \cap V(F_{(\beta)}(v_{\beta})) = \emptyset.
    \label{eq:F-i-1-disjoint-from-F-j-1}
  \end{equation}
  Recall that for all $A\in \mathcal{B}_{\alpha}$, $B\in\mathcal{B}_{\beta}$, we have that $\pi_1(A)\subseteq F_{(\alpha)}(v_{\alpha})$,  $\pi_1(B)\subseteq F_{(\beta)}(v_{\beta})$. Therefore by~\eqref{eq:F-i-1-disjoint-from-F-j-1},
  \begin{equation}
    V(\pi_1(A)) \cap V(\pi_1(B))=\emptyset, \textrm{ for all $A\in\mathcal{B}_\alpha$, $B\in\mathcal{B}_\beta$.}
    \label{eq:B-disjoin-from-B'}
  \end{equation}
  If $c=1$, then by the definition of the process $\mathcal{B}_z$ is non-empty for each $z\in[k]$. 
  Taking one element in $\mathcal{B}_z$ for each $z\in[k]$ yields an independent set (by~\eqref{eq:B-disjoin-from-B'}) of size $k$, which establishes \cref{item:hungarian-support:independent-set}.

  Now suppose $c\ge 2$. 
  For each $z\in[k]$, we fix an independent set $\mathcal{D}_z$ of $(F_i)_{i\in\set{2,\ldots,c}}$-tuples 
  contained in $\mathcal{C}_z(v_z)$
  with $|\mathcal{D}_{z}|\geq k^{c-1}$. By \cref{cor:interface-for-systems-lemma}, for each $z\in[k]$ there exists $D_z\in\mathcal{D}_z$ such that $\set{D_1,\ldots,D_k}$ is independent.
  For each $z\in[k]$, select some $A_z\in\mathcal{B}_z$ such that $(\pi_2(A_z),\ldots,\pi_c(A_z))=D_z$.   We now show that $\set{A_1,\ldots,A_k}$ is an independent set of $(F_i)_{i\in[c]}$-tuples. To this end, let $A:=A_\alpha$ and $B:=A_\beta$ for some $\alpha,\beta\in[k]$ with $\alpha < \beta$.  We must show that $V(\pi_i(A))\cap V(\pi_i(B))=\emptyset$, for each $i\in[c]$. Since $A\in\mathcal{B}_{\alpha}$ and $B\in\mathcal{B}_{\beta}$, by~\eqref{eq:B-disjoin-from-B'} we have that $V(\pi_1(A))\cap V(\pi_1(B))=\emptyset$. Since $(\pi_2(A),\ldots,\pi_c(A))=D_{\alpha}$ and $(\pi_2(B),\ldots,\pi_c(B))=D_{\beta}$, and since $D_\alpha$ and $D_\beta$ are independent, we conclude that $V(\pi_i(A)) \cap V(\pi_i(B))=\emptyset$, for each $i\in\{2,\ldots,c\}$. This proves that $\set{A_1,\ldots,A_k}\subseteq\mathcal{A}$ is an independent set, so~\eqref{item:hungarian-support:independent-set} holds.

  Now we consider the case in which $m<k$. 
  Let $X_1:=\{v_1,\ldots,v_m\}$.  
  First suppose that $c=1$. 
  Note that $|X_1|=m\le k = (\prod_{j=1}^{c-1}2^{j!})\, k^{c!}$ in this case.   
  We claim that $X_1$ satisfies \ref{item:hungarian-support:hitting-set_precise}.
  Suppose the opposite. 
  Then there exists some $A\in \mathcal{A}$ with $X_1 \cap V(\pi_1(A))=\emptyset$. 
  Let $v$ be the root of $\pi_1(A)$. 
  The proof splits into two cases depending on whether $\set{v_1,\ldots,v_m}$ contains an ancestor of $v$ in $F_1$ or not.  
  If no ancestor of $v$ lies in $\set{v_1,\ldots,v_m}$, then 
  in particular $r_1\not\in \set{v_1,\ldots,v_m}$ and 
  $A\in\mathcal{B}_m(r_1)$. 
  This is a contradiction with the fact that the process has terminated. 
  If $\set{v_1,\ldots,v_m}$ contains an ancestor of $v$ in $F_1$, 
  fix $z\in[m]$ such that $v_z$ is an ancestor of $v$ in $F_1$ and is closest to $v$ in $F_1$. 
  Consider the moment that the process has chosen $v_z$. We have $A\in\mathcal{B}_z(u)$, so $v_z$ was not a minimal choice for $v_z$, contradiction.

  Now assume $c\ge 2$. 
  For each $z\in[m]$, let $U_{z}$ be the set of (at most two) children of $v_z$ in $F_{(z)}$.
  Let $u\in U_z$. 
  Since $u \in V(F_{(z)})$ we know that $u\not\in\set{v_1,\ldots,v_z}$. 
  Since $u$ was not chosen to be $v_{z}$, we know that 
  $\mathcal{C}_{z}(u)$ has no independent set of size $k^{c-1}$. 
  Let $\mathcal{E}_{z}:=\bigcup_{u\in U_z}\mathcal{C}_{z}(u)$.  
  Therefore $\mathcal{E}_{z}$ does not contain an independent set of size $2k^{c-1}$ for each $z\in[m]$.

  Next, recall that $F_1$ is a tree with root $r_1$.  If $r_1 = v_m$, then define $\mathcal{E}_{m+1}:=\emptyset$. Otherwise, define $\mathcal{E}_{m+1}:=\mathcal{C}_{m+1}(r_1)$.  
  (Note that $\mathcal{C}_{m+1}(v)$ is defined for all $v\in V(F_{(z)}(r_1))$ so in particular for $v=r_1$.) 
  In the latter case, since the process for constructing $v_1,\ldots,v_m$ stopped at iteration $m$, $\mathcal{E}_{m+1}=\mathcal{C}_{m+1}(r_1)$ does not contain an independent set of size $k^{c-1}$. Therefore $\mathcal{A'}:=\bigcup_{z\in[m+1]}\mathcal{E}_z$ does not contain an independent set of size $(m+1)2k^{c-1} \le 2k^{c}$.

  Observe that $\mathcal{A}'$ is a set of non-trivial $(F_i)_{i\in\{2,\ldots,c\}}$-tuples.  Apply the inductive hypothesis with $c':=c-1$ and $k':=2k^{c}$ to the set $\mathcal{A}'$ of non-trivial $(F_i)_{i\in\{2,\ldots,c\}}$-tuples.  Since $\mathcal{A}'$ does not contain an independent set of size $k'$, this results in outcome \ref{item:hungarian-support:hitting-set_precise}.  That is, we obtain a sequence of sets $X_2,\ldots,X_c$ with $X_i\subseteq V(F_i)$ for each $i\in\{2,\ldots,c\}$ such that $\sum_{i=2}^c|X_i|\le\ell(2k^{c},c-1)$,  and, for each $(A_2,\ldots,A_c)\in\mathcal{A}'$, there exists some $i\in\{2,\ldots,c\}$ such that $V(A_i)\cap X_i\neq\emptyset$.

  We claim that $X_1,\ldots,X_c$ satisfy the conditions of outcome \ref{item:hungarian-support:hitting-set_precise}.

  First we show that, for each $(A_1,\ldots,A_c)\in\mathcal{A}$, there exists $i\in[c]$ such that $V(A_i)\cap X_i\neq\emptyset$.  If $(A_2,\ldots,A_c)\in \mathcal{A}'$ then, by the inductive hypothesis, $V(A_i)\cap X_i\neq\emptyset$ for some $i\in\{2,\ldots,c\}$.  If $(A_2,\ldots,A_c)\notin\mathcal{A}'$, then this is because there is no $z\in[m+1]$ such that $(A_2,\ldots,A_c)\in\mathcal{E}_{z}$. Let $v$ be the root of $A_1$.

  First, suppose that $v$ has no $F_1$-ancestor in $X_1$. 
  Observe that $r_1 \notin X_1$ in this case, and thus $F_{(m+1)}(r_1)$ and $\mathcal{C}_{m+1}(r_1)$ are defined. 
  Observe also that $A_1\not\subseteq F_{(m+1)}(r_1)$, which is the component of $F_{(m+1)}=F_1-X_1$ that contains $r_1$, because $(A_2, \dots, A_c)\notin\mathcal{E}_{m+1}=\mathcal{C}_{m+1}(r_1)$. 
  Since $v$ has no $F_1$-ancestor in $X_1$, $v$ is in $F_{(m+1)}(r_1)$. Since $A_1\not\subseteq F_{(m+1)}(r_1)$, it must be that $A_1$ contains a vertex in $X_1$, so $V(A_1)\cap X_1\neq\emptyset$, as required.  
  
  Next, suppose that $v$ has an $F_1$-ancestor in $X_1$, and let $v_z$ be the ancestor of $v$ in $X_1$ that is closest to $v$.  If $v=v_z$ then $V(A_1)\cap X_1\supseteq\{v\}\supsetneq\emptyset$, as required.  Otherwise $A_1$ is contained in $F_1(u)$ for some child $u$ of $v_z$.  Since $(A_2, \dots, A_c)\notin\mathcal{E}_{z}\supseteq \mathcal{C}_{z}(u)$, we have that $A_1\not\subseteq F_{(z)}(u)$. 
  Since $F_{(z)}(u)$ is a connected component of $F_1(u)-\{v_1,\ldots,v_{z-1}\}$, we deduce that $V(A_1)$ contains $v_i$ for some $i\in[z-1]$ so $V(A_1)\cap X_1\supseteq\{v_i\}\supsetneq\emptyset$.

  All that remains is to show is 
  that $|X_i|\le (\prod_{j=1}^{c-1}2^{j!})\, k^{c!}$ for each $i\in[c]$. 
  For $i=1$, we have $|X_1|=m\le k$, which satisfies our requirement since $c\ge 1$.  
  For $i\in\{2,\ldots,c\}$, recalling that $k'=2k^c$, by the inductive hypothesis we obtain 
  \[\textstyle
    |X_i|
     \le \left(\prod_{j=1}^{c-2}2^{j!}\right)(k')^{(c-1)!}
     = \left(\prod_{j=1}^{c-2}2^{j!}\right)(2k^c)^{(c-1)!}
     = \left(\prod_{j=1}^{c-1}2^{j!}\right)k^{c!} ,
  \]
  as desired.   
\end{proof}

Finally, we can complete the proof of \cref{thm:gyarfas-lehel-general} (restated here for convenience).

\hungarians*
\begin{proof}
  By adding a vertex $r$ adjacent to one vertex in each component of $F$, we may assume that $F$ is a tree rooted at $r$.  As in the proof of \cref{lem:hungarians-Fi-tuples}, we may further assume that $F$ is a binary tree.

  Let $\ell:\N^2\to\N$ be the function given by \cref{lem:hungarians-Fi-tuples}.

  The proof is by induction on $c$.  When $c=1$, $\mathcal{A}$ is a set of $(F_i)_{i\in[1]}$-tuples and the result follows, with $\ell^\star(k,1)=\ell(k,1)$, from \cref{lem:hungarians-Fi-tuples}.  Now assume that $c\ge 2$.

  Let $A\in\mathcal{A}$ and let $c'$ be the number of components of $A$. Thus $c'\in[c]$. Define $\Pi(A):=(A_1,\ldots,A_c)$ so that $A_1,\ldots,A_{c'}$ are the components of $A$ while $A_{c'+1},\ldots,A_c$ are the null graphs. (The ordering of the components $A_1,\ldots,A_{c'}$ is chosen arbitrarily.) 

  For each $\{i,j\}\in\binom{[c]}{2}$ and each $A\in\mathcal{A}$, define $A_{\{i,j\}}$ as the graph obtained from $A$ as follows.  If one or both of $\pi_i(\Pi(A))$ or $\pi_j(\Pi(A))$ is the null graph then $A_{\{i,j\}}:=A$. Otherwise $A_{\{i,j\}}$ is the graph obtained from $A$ by adding the shortest path in $F$ that has one endpoint in $\pi_i(\Pi(A))$ and the other endpoint in $\pi_j(\Pi(A))$.  In either case, $A_{\{i,j\}}$ has at most $c-1$ connected components.

  For each $\{i,j\}\in\binom{[c]}{2}$, let $\mathcal{A}_{i,j}:=\{A_{\{i,j\}}\mid A\in \mathcal{A}\}$.  Then $\mathcal{A}_{\{i,j\}}$ is a set of non-null subgraphs of $F$ each having at most $c-1$ components.  Now apply the inductive hypothesis to $\mathcal{A}_{\{i,j\}}$.  If outcome \cref{item:hungarians-pack} holds, then the $k$ pairwise vertex-disjoint members of $\mathcal{A}_{\{i,j\}}$ correspond to $k$ pairwise vertex-disjoint members of $\mathcal{A}$ and there is nothing more to do. Otherwise, outcome \cref{item:hungarians-hit} holds and we obtain a set $X_{\{i,j\}}$ with $|X_{\set{i,j}}|\leq \ell^\star(k,c-1)$ such that $X_{\{i,j\}}\cap V(A_{\{i,j\}})\neq\emptyset$ for each $A\in\mathcal{A}$.

  Let $X_0:=\bigcup_{\{i,j\}\in\binom{[c]}{2}} X_{\{i,j\}}$, let $F_0:=F-X_0$ and let $\mathcal{A}_0:=\{A\in\mathcal{A}\mid V(A)\cap X_0=\emptyset\}$.  Since $F$ is a binary tree, the number $m$ of components of $F_0$ is
  \begin{equation}
    \textstyle m \leq 1+2|X_0|\le 1+2\binom{c}{2}\,\ell^\star(k,c-1). 
  \end{equation}
  Let $F_1,\ldots,F_m$ denote the components of $F_0$. Consider some $A\in\mathcal{A}_0$ having two distinct connected components $A_i:=\pi_i(\Pi(A))$ and $A_j:=\pi_j(\Pi(A))$.  Since $V(A_{\{i,j\}})\cap X_{\{i,j\}}\neq\emptyset$ and $V(A)\cap X_{\{i,j\}}=\emptyset$, we conclude that $X_{\set{i,j}}$ contains a vertex in $V(A_{\set{i,j}})\setminus V(A)$. Therefore, $A_i$ and $A_j$ are in different components of $F_0$.  Thus, each component of $F_0$ contains at most one component of $A$ for each $A\in\mathcal{A}_0$.

  For each $A\in\mathcal{A}_0$, define $\Pi_0(A)$ as the $(F_i)_{i\in[m]}$-tuple in which $\pi_i(\Pi_0(A))$ is the component of $A$ contained in $F_i$ or the null graph if $V(A)\cap V(F_i)=\emptyset$, for each $i\in[m]$.  Since $A$ is non-null, $\Pi_0(A)$ is non-trivial. At this point, we can immediately apply \cref{lem:hungarians-Fi-tuples} with $c':=m$ to the set $\{\Pi_0(A):A\in \mathcal{A}_0\}$ of non-trivial $(F_i)_{i\in[m]}$-tuples.  If outcome \cref{item:hungarian-support:independent-set} occurs then the resulting set of $k$ independent $(F_i)_{i\in[m]}$-tuples corresponds to a set of pairwise vertex-disjoint elements of $\mathcal{A}_0\subseteq \mathcal{A}$, which satisfies the conditions of \cref{item:hungarians-pack}.  If outcome \cref{item:hungarian-support:hitting-set} occurs, then we obtain sets $X_1,\ldots,X_m$ with $X_i\subseteq V(F_i)$, for each $i\in[m]$, of total size at most $\ell(k,m)$ such that for each $A\in\mathcal{A}_0$, there exists $i\in[m]$ such that $V(A)\cap X_i\neq\emptyset$.  Then $X:=X_0\cup \bigcup_{i\in[m]}X_i$ satisfies the conditions of \cref{item:hungarians-hit} with
  \[
     \textstyle |X|\le |X_0|+\sum_{i=1}^m |X_i| \le \binom{c}{2}\,\ell^\star(k,c-1) + \ell(k,m)
     \le \binom{c}{2}\,\ell^\star(k,c-1) + \ell(k,2\binom{c}{2}\,\ell^\star(k,c-1)). 
  \]
  This proves that
  \[  \textstyle \ell^\star(k,c) \le \binom{c}{2}\,\ell^\star(k,c-1) + \ell(k,2\binom{c}{2}\,\ell^\star(k,c-1)) \]
  which is sufficient to establish \cref{thm:gyarfas-lehel-general}.  However, with a bit more work, one can prove a bound of the form $\ell^\star(k,c)\in O(k^{h(c)})$ for some $h:\N\to\N$, which is polynomial in $k$ for any fixed $c$.

  For each $I\in\binom{[m]}{c}$, define the set $\mathcal{A}_I:=\{A\in\mathcal{A}_0:A\subseteq \bigcup_{i\in I}F_i\}$.  Since each element in $A$ has at most $c$ components and $V(A)\subseteq V(F_0)$ for each $A\in \mathcal{A}_0$, each such $A$ appears in $\mathcal{A}_I$ for at least one $I\in\binom{[m]}{c}$.  For each $I\in\binom{[m]}{c}$ and each $A\in\mathcal{A}_I$, define $\Pi_I(A):=(\pi_i(\Pi_0(A))_{i\in I}$ as the $(F_i)_{i\in I}$-tuple obtained by restricting $\Pi_0(A)$ to the indices in $I$.  For each $I\in\binom{[m]}{c}$, apply \cref{lem:hungarians-Fi-tuples} to the set $\{\Pi_i(A):A\in \mathcal{A}_I\}$ of $(F_i)_{i\in I}$-tuples.  If outcome \cref{item:hungarian-support:independent-set} occurs for some $I\in\binom{[m]}{c}$ then the resulting set of $k$ independent $(F_i)_{i\in I}$-tuples corresponds to a set of pairwise vertex-disjoint elements of $\mathcal{A}_I\subseteq \mathcal{A}$, which satisfies the conditions of \cref{item:hungarians-pack}.  If outcome \cref{item:hungarian-support:hitting-set} occurs for all $I\in\binom{[m]}{c}$, then we obtain a set $X_I\subseteq V(F_0)\subseteq V(F)$ with $|X_I|\le \ell(k,c)$ such that $V(A)\cap X_I\neq\emptyset$ for each $A\in\mathcal{A}_I$.

  Define the set $X:= X_0 \cup \bigcup_{I\in\binom{[m]}{c}} X_I$. As discussed above, each $A\in \mathcal{A}\setminus \mathcal{A}_0$ contains a vertex of $X_0$ while each $A\in\mathcal{A}_0$ contains a vertex of $\bigcup_{I\in\binom{[m]}{c}} X_I$. Finally,
  \[
      |X|\le |X_0|+\bigcup_{I\in\binom{[m]}{c}}|X_I|
     \le \binom{c}{2}\,\ell^\star(k,c-1) + \binom{m}{c}\ell(k,c)
     \le \binom{c}{2}\,\ell^\star(k,c-1) + \binom{2\binom{c}{2}\ell^\star(k,c-1)}{c}\ell(k,c), 
  \]
  so~\cref{item:hungarians-hit} holds.
\end{proof}

By expanding the functions $\ell$ and $\ell^\star$ that appear in the proofs above one can see that $\ell(k,2)=\Oh(k^2)$, $\ell(k,3)=\Oh(k^6)$, 
$\ell^{\star}(k,1)=k$, 
$\ell^{\star}(k,2)=\Oh(k^4)$, $\ell^{\star}(k,3)=\Oh(k^{18})$.  For any fixed $c\ge 1$, $\ell^{\star}(k,c)=\Oh(k^{c\cdot c!})$.

\section{Discussion}

We conjecture that \cref{thm:main-in-intro} remains true with $f(k)\in\Oh(k \log k)$ as in the original Erd\H{o}s-P\'osa theorem, which would be optimal. Our proof gives $f(k)\in\Oh(k^{18}\polylog k)$. By applying the techniques in \cref{sec:hungarians} directly to the tuples $\{(\Psi_\ell(t))_{\ell\in\{0,\ldots,p\}}:\text{$t$ is an admissible tuple}\}$ that appear in the proof of \cref{thm:the-big-ball-of-wax}, the exponent $18$ could be reduced with some additional work but would still be much greater than $1$.

Turning to the function $g(d)$ in  \cref{thm:main-in-intro}, we remark that we made no attempt to minimize the leading constant in our proof, which shows $g(d)=19d$. We believe that we can decrease the constant $19$ somewhat, at the price of a more complicated proof, but it would still be above $10$. It would be interesting to determine the smallest function $g(d)$ for which \cref{thm:main-in-intro} is true (with no constraint on $f(k)$).

\section*{Acknowledgements}

The authors would like to thank Ugo Giocanti, Jędrzej Hodor, Freddie Illingworth, and Clément Legrand-Duchesne for discussions in the early stage of this research, Raj Kaul for spotting typos in an earlier version of the manuscript, and the two anonymous reviewers for their very helpful comments.  
This research was initiated at the Third Ottawa-Carleton Workshop on Graphs and Probability, held in Mont Sainte-Marie, Québec and in Toronto, Ontario, Oct 13–20, 2024. 
Partially completed at the Thirteenth Annual Workshop on Geometry and Graphs held in Holetown, Barbados, Jan 30-Feb 6, 2026.

\bibliographystyle{plainurlnat}
\bibliography{cep}

\begin{thebibliography}{6}
\providecommand{\natexlab}[1]{#1}
\providecommand{\url}[1]{\texttt{#1}}
\providecommand{\urlprefix}{URL }
\expandafter\ifx\csname urlstyle\endcsname\relax
  \providecommand{\doi}[1]{\href{https://dx.doi.org/#1}{\nolinkurl{doi:#1}}}\else
  \providecommand{\doi}[1]{\href{https://dx.doi.org/#1}{\nolinkurl{doi:#1}}}\fi
\providecommand{\eprint}[2][]{\url{#2}}

\bibitem[{Ahn et~al.(2025)Ahn, Gollin, Huynh, and Kwon}]{ahn.gollin:coarse}
Jungho Ahn, J.~Pascal Gollin, Tony Huynh, and O{-}joung Kwon.
\newblock A coarse {E}rdős-{P}ósa theorem.
\newblock In \emph{ACM-SIAM Symposium on Discrete Algorithms (SODA25)}. 2025.
\newblock To appear.

\bibitem[{Erd{\H{o}}s and P{\'o}sa(1965)}]{EP1965}
Paul Erd{\H{o}}s and Lajos P{\'o}sa.
\newblock On independent circuits contained in a graph.
\newblock \emph{Canadian Journal of Mathematics}, 17:347--352, 1965.

\bibitem[{Georgakopoulos and Papasoglu(2023)}]{GP23}
Agelos Georgakopoulos and Panos Papasoglu.
\newblock Graph minors and metric spaces, 2023.
\newblock \href{https://arxiv.org/abs/2305.07456}{arXiv:2305.07456}.

\bibitem[{Gyárfás and Lehel(1970)}]{gyarfas.lehel:helly}
András Gyárfás and Jenö Lehel.
\newblock A {H}elly-type problem in trees.
\newblock \emph{Combinatorial Theory and its applications}, 4:571--584, 1970.

\bibitem[{Robertson and Seymour(1986)}]{RS1986}
Neil Robertson and Paul~D. Seymour.
\newblock Graph minors. {V}. {E}xcluding a planar graph.
\newblock \emph{Journal of Combinatorial Theory, Series B}, 41(1):92--114,
  1986.

\bibitem[{Simonovits(1967)}]{Simonovits67}
Mikl{\'o}s Simonovits.
\newblock A new proof and generalizations of a theorem of {E}rdős and
  {P}{\'o}sa on graphs without $k+1$ independent circuits.
\newblock \emph{Acta Mathematica Academiae Scientiarum Hungarica}, 18:191--206,
  1967.

\end{thebibliography}

\end{document}